\documentclass{amsart}
\usepackage{color}

\usepackage[utf8]{inputenc}
\usepackage{amsmath,amssymb,amsbsy,amsfonts,latexsym}
\usepackage{verbatim}

\newtheorem{thm}{Theorem}
\newtheorem{corollary}[thm]{Corollary}

\newtheorem{definition}{Definition}[section]
\newtheorem{lemma}[thm]{Lemma}
\newtheorem{proposition}[thm]{Proposition}
\newtheorem{remark}{Remark}[section]

\newcommand{\dd}{\textup{d}}
\newcommand{\R}{\mathbb{R}}
\newcommand{\C}{\mathbb{C}}
\newcommand{\Z}{\mathbb{Z}}

\begin{document}
\title[Magnetic Choquard equation with Hardy-Littlewood-Sobolev critical exponent.]{Nonlinear Perturbations of a periodic magnetic Choquard equation with Hardy-Littlewood-Sobolev critical exponent}
	
\author{H. Bueno, N. H. Lisboa and L. L. Vieira }
	
\begin{abstract}
In this paper, we consider the following magnetic nonlinear Choquard equation
\[-(\nabla+iA(x))^2u+ V(x)u = \left(\frac{1}{|x|^{\alpha}}*|u|^{2_{\alpha}^*}\right) |u|^{2_{\alpha}^*-2} u + \lambda f(u)\ \textrm{ in }\ \R^N,\]
where $2_{\alpha}^{*}=\frac{2N-\alpha}{N-2}$ is the critical exponent in the sense of the Hardy-Littlewood-Sobolev inequality, $\lambda>0$, $N\geq 3$, $0<\alpha< N$, $A: \mathbb{R}^{N}\rightarrow \mathbb{R}^{N}$ is an $C^1$, $\mathbb{Z}^N$-periodic vector potential and $V$ is a continuous scalar potential given as a perturbation of a periodic potential. Under suitable assumptions on different types of nonlinearities $f$, namely,  $f(x,u)=\left(\frac{1}{|x|^{\alpha}}*|u|^{p}\right)|u|^{p-2} u$ for $(2N-\alpha)/N<p<2^{*}_{\alpha}$, then $f(u)=|u|^{p-1} u$ for $1<p<2^*-1$ and $f(u)=|u|^{2^* - 2}u$ (where $2^*=2N/(N-2)$), we prove the existence of at least one ground state solution for this equation by variational methods if $p$ belongs to some intervals depending on $N$ and $\lambda$.

\end{abstract}
	
\maketitle
	
	

\section{Introduction}
	
In this article we consider the problem
\begin{equation}\label{main}
-(\nabla+iA(x))^2u+ V(x)u = \left(\frac{1}{|x|^{\alpha}}*|u|^{2_{\alpha}^*}\right) |u|^{2_{\alpha}^*-2} u + \lambda f(u)\ \textrm{ in }\ \R^N,
\end{equation}
where $\nabla+iA(x)$ is the covariant derivative with respect to the $C^1$, $\Z^N$-periodic vector potential $A\colon \mathbb{R}^{N}\rightarrow \mathbb{R}^{N}$, i.e,
\begin{equation*}
A(x+y)=A(x),\ \forall\; x\in 	\R^N,\ \forall\; y\in\Z^N.
\end{equation*}
The exponent  $2_{\alpha}^{*}=\frac{2N-\alpha}{N-2}$ is critical, in the sense of the Hardy-Littlewood-Sobolev inequality, $\lambda>0$, $N\geq 3,$ $0<\alpha< N,$  $V:\R^N\to \R$ is a continuous scalar potential and $f$ stands for different types of nonlinearities. Namely, we first consider  $f(x,u)=\left(\frac{1}{|x|^{\alpha}}*|u|^{p}\right)|u|^{p-2} u$ for $(2N-\alpha)/N<p<2^{*}_{\alpha}$, then $f(u)=|u|^{p-1} u$ for $1<p<2^*-1$, where $2^*$ is the critical exponent of immersion $D^{1,2}(\mathbb{R}^N)\hookrightarrow L^{2^*}(\mathbb{R}^N)$,  and finally  we examine $f(u)=|u|^{2^* - 2}u$. 

Inspired by the seminal work of Coti Zelati and Rabinowitz \cite{CotiZelati}, but also by Alves, Carrião and Miyagaki \cite{Claudionor} and by Alves and Figueiredo \cite{Giovany}, we assume that there is a continuous, $\Z^N$-periodic potential  $V_{\mathcal{P}}:\R^N\to \R$,  constants $V_0,W_0>0$ and $W\in L^{\frac{N}{2}}(\R^N)$ with  $W(x)\geq 0$ such that
\begin{itemize}
\item[$(V_1)$] $V_{\mathcal{P}}(x)\geq V_0,\quad\forall\; x\in \R^N$;
\item[$(V_2)$]	$V(x)=V_{\mathcal{P}}(x)-W(x)\geq W_0,\quad\forall\; x\in\R^N$,	
\end{itemize}
where the last inequality is strict on a subset of positive measure in $\R^N$.

Since the problem is considered in  the whole $\R^N$ and has a critical nonlinearity in the Hardy-Littlewood-Sobolev sense, the verification of any compactness condition is not easy.

Our paper is motivated by Gao and Yang in \cite{BrezisN1}, where a classical Choquard equation is considered in a bounded domain, i.e., the case $A\equiv 0$ and $V\equiv 0$ is studied in a bounded domain $\Omega$. There is a huge literature about the Choquard equation and we cite only Moroz and Van Schaftingen \cite{Moroz2} for a good review of results on this important subject. In \cite{BrezisN1}, Gao and Yang proved the existence of a ground state solution under restriction on $N$ and $\lambda$. Other recent advances in the study of the Choquard equation can be found, e.g., in \cite{AlvesMinbo4, AlvesMinbo3, AlvesMinbo1,LeleMinbo2, LeleMinbo,Minbo2, Moroz1,ShenYang}.

In Mukherjee and Sreenadh \cite{Mukherjee}, the magnetic problem
\[-(\nabla+iA(x))^2u+ \mu g(x)u = \ \lambda u+\left(\frac{1}{|x|^{\alpha}}*|u|^{2_{\alpha}^*}\right) |u|^{2_{\alpha}^*-2} u\ \textrm{ in }\ \R^N\]
was examined. In this equation $\mu>0$ is also a parameter that interacts with the linear term in the right-hand side of the equation. Under suitable hypotheses on $g$, the existence of a ground state solution was proved. The concentration of solutions as $\mu\to \infty$ was also studied.

Changing the right-hand side of \eqref{main} to
\begin{equation}\label{Cingolani}\bigg(\frac{1}{|x|^{\alpha}}*|u|^p)\bigg)|u|^{p-2}{u},
\end{equation}
the problem was studied by Cingolani, Clapp and Secchi in \cite{Cingolani}. In that paper the authors proved existence and multiplicity of solutions. In \cite{Hamilton}, the right-hand side \eqref{Cingolani} was generalized and a ground state solution was obtained, but the multiplicity result depend on more restrictive hypotheses than in \cite{Cingolani}.

Recent years have witnessed a growth of interest in the study of magnetic equations. The progress in this research can be found in a series of articles, e.g., \cite{AFY,AlvesMinbo2,Ambrosiod'Avenia,Ambrosio,d'Avenia,d'AveniaSqua}.

The main results of this paper are the following theorems.

\begin{thm}\label{teoprob1}
For $\frac{2N-\alpha}{N}<p<2_{\alpha}^*$, under the hypotheses already stated on $A$, $V$ and $\alpha$, problem
\begin{equation}\label{prob1}
-(\nabla+iA(x))^2u+ V(x)u =\left(\frac{1}{|x|^{\alpha}}*|u|^{2_{\alpha}^*})\right)|u|^{2_{\alpha}^*-2} u + \lambda\left(\frac{1}{|x|^{\alpha}}*|u|^{p})\right)|u|^{p-2} u\ \textrm{ in }\ \R^N
\end{equation} has at least one ground state solution if either
\begin{enumerate}
\item [$(i)$] $\frac{N+2-\alpha}{N-2}<p<2_{\alpha}^*$, $N=3,4$ and $\lambda>0$;
\item [$(ii)$] $\frac{2N-\alpha}{N}<p\leq \frac{N+2-\alpha}{N-2}$,  $N=3,4$ and $\lambda$ sufficiently large;
\item [$(iii)$] $\frac{2N-2-\alpha}{N-2}<p<2_{\alpha}^*$, $N\geq 5$ and $\lambda>0$;
\item [$(iv)$] $\frac{2N-\alpha}{N}<p\leq \frac{2N-2-\alpha}{N-2}$,  $N\geq 5$ and $\lambda$ sufficiently large.
\end{enumerate}
\end{thm}

\begin{thm}\label{teoprob2}
For $1<p<2^{*} - 1$, under the hypotheses already stated on $A$, $V$ and $\alpha$,  problem
\begin{equation}\label{prob2}
-(\nabla+iA(x))^2u+ V(x)u = \ \bigg(\frac{1}{|x|^{\alpha}}*|u|^{2_{\alpha}^*}\bigg)|u|^{2_{\alpha}^*-2} u + \lambda |u|^{p-1}u\  \textrm{ in }\ \R^N.
\end{equation}
has at least one ground state solution if either
\begin{enumerate}
\item [$(i)$] $3<p<5$, $N=3$ and $\lambda>0;$
\item [$(ii)$] $p>1$, $N\geq 4$ and $\lambda>0$;
\item [$(iii)$] $1<p\leq 3$, $N=3$ and $\lambda$ sufficiently large.
\end{enumerate}
\end{thm}

\begin{thm}\label{teoprob3}
Under the hypotheses already stated on $A$, $V$ and  $\alpha$, the problem
\begin{equation}\label{prob3}
-(\nabla+iA(x))^2u+ V(x)u = \ \lambda\left(\frac{1}{|x|^{\alpha}}*|u|^{p}\right)|u|^{p-2} u+ |u|^{2^* - 2} u\ \textrm{ in }\ \R^N,
\end{equation}
has  at least one ground state solution in the  intervals already described in Theorem \ref{teoprob1}.
\end{thm}

Initially, we are going to prove the existence of a ground state solution for problem \eqref{main} considering the potential $V=V_{\mathcal{P}}$, that is, we consider the problem
\begin{equation}\label{prop}
-(\nabla+iA(x))^2u+ V_{\mathcal{P}}(x)u =\left(\frac{1}{|x|^{\alpha}}*|u|^{2_{\alpha}^*})\right)|u|^{2_{\alpha}^*-2} u + \lambda f(u)\ \textrm{ in }\ \R^N
\end{equation}
and $f$ as in Theorems \ref{teoprob1}, \ref{teoprob2} and \ref{teoprob3},
where we maintain the notation introduced before and suppose that $(V_1)$ is valid.
	
As in  Gao and Yang in \cite{BrezisN1}, the key step to proof the existence of a ground state solution of problem \eqref{prop} is the use of cut-off techniques on the extreme function that attains the best constant $S_{H,L}$ defined in the sequence. This allows us to estimate the mountain pass value $c_{\lambda}$ associated with the energy functional $J_{A,V_{\mathcal{P}}}$ related with \eqref{prop} in terms of the level where the PS condition holds. In a demanding proof, this lead us to establish intervals for $p$ (depending on $N$ and $\lambda$) where the PS condition is satisfied, as in the seminal work of Brézis and Nirenberg \cite{BrezisNiremberg}. After that, the proof is completed by showing the mountain pass geometry, introducing the Nehari manifold associated with \eqref{prop} and applying concentration-compactness arguments. In the sequel, we consider \eqref{main} for the different nonlinearities $f$ and prove that each problem has at least one ground state solution.

We observe that the conclusion of Theorem \ref{teoprob2} is similar to that of Theorem 1.1 in Alves, Carrião and Miyagaki \cite{Claudionor}  and  Theorem 1.1 in Miyagaki \cite{Olimpio}. Being more precise, in \cite{Claudionor} the authors have discussed the existence of a positive solution to the semilinear elliptic problem involving critical exponents
\[-\Delta u + V(x) u= \lambda u^{q} +  u^{p}\ \textrm{ in }\ \R^N,\]
where $\lambda>0$ is a parameter, $1<q<p=2^*-1$ and $V:\R^N\to\R$ is a positive continuous function. On its turn, Miyagaki \cite{Olimpio}   has studied the existence of nontrivial solution for the following class of semilinear
elliptic equation in $\R^N$ ($N\geq 3$) involving critical Sobolev exponents
\[-\Delta u + a(x) u= \lambda |u|^{q-1} +  |u|^{p-1} u \ \textrm{ in }\ \R^N,\]
where $1<q<p\leq 2^*-1=\frac{N+2}{N-2}$ and $\lambda>0$ are constants and $a:\R^N\to \R$ is a
continuous function such that $a(x)\geq a_0$ for all $x\in\R^N$, where $a_0>0$ is a constant.

Problems \eqref{prop} and \eqref{main} are then related by showing that the minimax value $d_\lambda$ of the latter satisfies $d_\lambda<c_\lambda$. Once more, concentration-compactness arguments are applied to show the existence of a ground state solution.

This paper is organized as follows. In Section 2 some preliminary results will be established. Section 3, 4 and 5 are then devoted to the proofs of Theorems  \ref{teoprob1},  \ref{teoprob2} and  \ref{teoprob3}, respectively. 	
\section{Preliminary results}
We denote
\[\nabla_{A} u=\nabla u+iA(x)u.\]

We handle problem \eqref{main} in the space
\[H^1_{A,V}(\R^N,\C)=\left\{u\in L^2(\mathbb{R}^N,\mathbb{C})\,:\, \nabla_{A} u\in L^2(\mathbb{R}^N,\mathbb{C}),\ \ \int_{\mathbb{R}^N}V(x) |u(x)|^2\;\dd x<\infty\right\}\]
endowed with the norm
\[\|u\|_{A,V}=\bigg(\int_{\R^N}(|\nabla_{A} u|^2+V(x)|u|^2)\;\dd x \bigg)^{\frac{1}{2}}.\]

Observe that the norm generated by this scalar product is equivalent to the norm obtained by considering $V\equiv 1$, see \cite[Definition 7.20]{LiebLoss}.

If $u\in H^1_{A,V}(\mathbb{R}^N,\mathbb{C})$, then $|u|\in H^1(\mathbb{R}^N)$ and the \emph{diamagnetic inequality} is valid (see \cite{Cingolani} or \cite[Theorem 7.21]{LiebLoss})
\[|\nabla |u|(x)|\leq |\nabla u(x)+iA(x)u(x)|,\ \ \textrm{a.e. } x\in\mathbb{R}^N.\]

As a consequence of the diamagnetic inequality, we have the continuous immersion
\begin{equation}\label{immersion}H^1_{A,V}(\mathbb{R}^N,\mathbb{C})\hookrightarrow L^s(\mathbb{R}^N,\mathbb{C})\end{equation}
for any $s\in [2,\frac{2N}{N-2}]$. We denote $2^*=\frac{2N}{N-2}$ and $\|\cdot\|_s$ the norm in $L^s(\R^N,\C)$.

It is well-known that $C^\infty_c(\mathbb{R}^N,\mathbb{C})$ is dense in $H^1_{A,V_{\mathcal{P}}}(\mathbb{R}^N,\mathbb{C})$, see \cite[Theorem 7.22]{LiebLoss}.

Following Gao and Yang \cite{Minbo1}, we denote by $S_{H,L}$
\begin{align}\label{constantesobol}
S_{H,L}:
&=\inf_{u\;\in\; D^{1,2}(\R^N,\R)\setminus \{0\}}\frac{\displaystyle \int_{\R^N}|\nabla u|^2 \dd x}{\left(\displaystyle\int_{\R^N}\int_{\R^N}\frac{|u(x)^{2_{\alpha}^*}|\ |u(y)|^{2_{\alpha}^*}}{|x-y|^{\alpha}}\dd x\dd y\right)^{\frac{N-2}{2N-\alpha}}}\\
&=\inf_{u\in D^{1,2}_A(\R^N)\setminus\{0\}}\frac{\displaystyle\int_{\R^N}|\nabla_A u|^2\dd x}{\left(\displaystyle\int_{\R^N}\int_{\R^N}\frac{|u(x)^{2_{\alpha}^*}|\ |u(y)|^{2_{\alpha}^*}}{|x-y|^{\alpha}}\dd x\dd y\right)^{\frac{N-2}{2N-\alpha}}}=:S_A,\nonumber
\end{align}
where $D^{1,2}_A(\R^N)=\{u\in L^{2^*}(\R^N,\C)\})\,:\,\nabla_A u\in L^2(\R^N,\C)\}$. The equality between $S_{H,L}$ and $S_A$ was proved in Mukherjee and Sreenadh \cite{Mukherjee}. We remark that $S_A$ is attained if and only if $\textup{rot}\, A=0$ \cite[Theorem 4.1]{Mukherjee}. See also \cite[Theorem 1.1]{ArioliSzulkin}.
	
We state a result proved in \cite{Minbo1}.
\begin{proposition}[Gao and Yang \cite{Minbo1}]\label{propsobol} The constant $S_{H,L}$ defined in \eqref{constantesobol} is achieved if and only if
\[u=C\left(\frac{b}{b^2+|x-a|^2}\right)^{\frac{N-2}{2}},\]
where $C>0$ is a fixed constant, $a\in\R^N$ and $b\in (0,\infty)$ are parameters. Furthermore,
\[S_{H,L}=\frac{S}{C (N,\alpha)^{\frac{N-2}{2N-\alpha}}},\]
where $S$ is the best Sobolev constant of the immersion $D^{1,2}(\R^N)\hookrightarrow L^{2^*}(\R^N)$ and $C (N,\alpha)$ depends on $N$ and $\alpha$.
\end{proposition}
	
If we consider the minimizer for $S$ given by $U(x):=\frac{[N(N-2)]^{\frac{N-2}{4}}}{(|1+|x|^2|)^{\frac{N-2}{2}}}$ (see \cite[Theorem 1.42]{Willem}), then
\[\bar{U}(x)=S^{\frac{(N-\alpha)(2-\alpha)}{4(N+2-\alpha)}}C(N,\alpha)^{\frac{2-N}{2(N+2-\alpha)}}\frac{[N(N-2)]^{\frac{N-2}{4}}}{(|1+|x|^2|)^\frac{N-2}{2}}\]
is the unique minimizer for $S_{H,L}$ that satisfies
\[-\triangle u=\left(\int_{\R^N}\frac{|u|^{2_{\alpha}^*}}{|x-y|^{\alpha}}\dd y\right)|u|^{2_{\alpha}^*-2}u\;\;\textrm{em}\;\; \R^N\]
with
\[\displaystyle \int_{\R^N}|\nabla \bar{U}|^2 \dd x=\displaystyle \int_{\R^N}\int_{\R^N}\frac{|\bar{U}(x)|^{2_{\alpha}^*} |\bar{U}(y)|^{2_{\alpha}^*}}{|x-y|^{\alpha}}\dd x\dd y=S_{H,L}^{\frac{2N-\alpha}{N+2-\alpha}}.\]

\begin{proposition}[Hardy-Littlewood-Sobolev inequality, see \cite{LiebLoss}]\label{hls} Suppose that $f\in L^t(\R^N)$ and $h \in L^r(\R^N)$ for $t,r>1$ and $0<\alpha<N$ satisfying $\frac{1}{t}+\frac{\alpha}{N}+\frac{1}{r}=2$. Then, there exists a sharp constant  $C(t,N,\alpha,r)$, independent of $f$ and $h$, such that
\begin{equation}\label{desig}
\int_{\R^N}\int_{\R^N}\frac{f(x) h(y)}{|x-y|^{\alpha}}\dd x\dd y\leq C(t,N,\alpha,r)\|f\|_t \|h\|_r.
\end{equation}
If $t=r=\frac{2N}{2N-\alpha}$, then \[C(t,N,\alpha,r)=C(N,\alpha)=\pi^{\frac{\alpha}{2}}\frac{\Gamma(\frac{N}{2}-\frac{\alpha}{2})}{\Gamma(N-\frac{\alpha}{2})}\left\{\frac{\Gamma(\frac{N}{2})}{\Gamma(N)}\right\}^{-1+\frac{\alpha}{N}}.\]
In this case there is equality in $\eqref{desig}$ if and only if $h=cf$ for a constant $c$ and
\[f(x)=A(\gamma^2+|x-a|^2)^{-(2N-\alpha)/2}\]
for some $A\in \C$, $0\neq \gamma \in \R$ and $a\in \R^N$.
\end{proposition}

\begin{lemma}\label{lK}
Let $U\subseteqq \mathbb{R}^{N}$ be any open set. For $1<p<\infty$, let $(f_n)$ be a bounded sequence in $L^s(U,\mathbb{C})$ such that $f_n(x)\to f(x)$ a.e. Then $f_n\rightharpoonup f$ in $L^s(U,\C)$.
\end{lemma}

The proof of Lemma \ref{lK} only adapts the arguments given for the real case, as in \cite[Lemme 4.8, Chapitre 1]{Kavian}.

\section{The case $f(x,u)=\left(\frac{1}{|x|^{\alpha}}*|u|^{p}\right)|u|^{p-2} u$}\label{Section3}

\subsection{The periodic problem}\label{periodic} In this subsection we deal with problem \eqref{prop} for $f(x,u)$ as above, that is,
\begin{equation}\label{prop1.1}
	-(\nabla+iA(x))^2u+  V_{\mathcal{P}}(x)u =\left(\frac{1}{|x|^{\alpha}}*|u|^{2_{\alpha}^*})\right)|u|^{2_{\alpha}^*-2} u + \lambda\left(\frac{1}{|x|^{\alpha}}*|u|^{p})\right)|u|^{p-2} u,
\end{equation}
where $\frac{2N-\alpha}{N}<p<2_{\alpha}^*$.

We consider the space
\[H^1_{A,V_{\mathcal{P}}}(\mathbb{R}^N,\mathbb{C})=\left\{u\in L^2(\mathbb{R}^N,\mathbb{C})\,:\, \nabla_{A} u\in L^2(\mathbb{R}^N,\mathbb{C})\right\}\]
endowed with scalar product
\[\langle u,v\rangle_{A,V_{\mathcal{P}}}=\mathfrak{Re}\int_{\mathbb{R}^N}\left(\nabla_A u\cdot \overline{\nabla_A v}+V_{\mathcal{P}}(x)u\bar v\right)\dd x\]
and, therefore
\[\|u\|^2_{A,V_{\mathcal{P}}}=\int_{\mathbb{R}^N}\left(|\nabla_A u|^2+V_{\mathcal{P}}|u|^2\right)\dd x.\]

We observe that the energy functional $J_{A,V_{\mathcal{P}}}$  on $H^1_{A,V_{\mathcal{P}}}(\R^N,\C)$ associated with \eqref{prop1.1} is given by
\begin{equation*}
J_{A,V_{\mathcal{P}}}(u):=\frac{1}{2}\Vert u\Vert^2_{A,V_{\mathcal{P}}}-\frac{1}{2\cdot 2_{\alpha}^*}D(u)-\frac{\lambda}{2p}B(u),
\end{equation*}
where
\[B(u)=\int_{\R^N}\bigg(\frac{1}{|x|^{\alpha}}*|u|^p\bigg)|u|^p\dd x= \int_{\R^N}\int_{\R^N}\frac{|u(x)^p| |u(y)|^p}{|x-y|^{\alpha}}\dd x\dd y\]
and
\[D(u)=\int_{\R^N}\bigg(\frac{1}{|x|^{\alpha}}*|u|^{2_{\alpha}^*}\bigg)|u|^{2_{\alpha}^*}\dd x=\int_{\R^N}\int_{\R^N}\frac{|u(x)^{2_{\alpha}^*}| |u(y)|^{2_{\alpha}^*}}{|x-y|^{\alpha}}\dd x\dd y.\]

\begin{remark}\label{obs3}\rm
Notice that, by the Hardy-Littlewood-Sobolev inequality, the integral
\[\displaystyle \int_{\R^N}\int_{\R^N}\frac{|u(x)^s| |u(y)|^s}{|x-y|^{\alpha}}\dd x\dd y\]
is well defined if
\begin{align*} \frac{2N-\alpha}{N}\leq s\leq\frac{2N-\alpha}{N-2}.
\end{align*}
Here, as also in \cite{AlvesMinbo1}, $\frac{2N-\alpha}{N}$ is called the lower critical exponent and $2_{\alpha}^*=\frac{2N-\alpha}{N-2}$ the upper critical exponent. This lead us to say that \eqref{main} is a critical nonlocal elliptic equation.
\end{remark}

The Hardy-Littlewood-Sobolev inequality implies that
\begin{equation}\label{desigb}
|B(u)|\leq C_1(N,\alpha)\|u\|^{2p}_{p}
\end{equation}
and
\begin{equation}\label{desigd}
|D(u)|\leq C_2(N,\alpha)\|u\|^{2\cdot 2_{\alpha}^*}_{2_{\alpha}^*}
\end{equation}
for constants $C_1(N,\alpha)$ and $C_2(N,\alpha)$. For any $u \in H^1_{A,V_{\mathcal{P}}}(\R^N,\C)$ the immersions \eqref{immersion} imply that $B$ and $D$ are well-defined. Consequently, $J_{A,V_{\mathcal{P}}}$ is well-defined.

Observe that 
\begin{equation}\label{dunshl}
S_{A}=\inf_{u\;\in\; D^{1,2}_A(\R^N)\setminus \{0\}}\frac{\int_{\R^N}|\nabla_A u|^2 \dd x}{D(u)^{\frac{N-2}{2N-\alpha}}}.
\end{equation}
	
\begin{definition}
A function $u\in H^1_{A,V_{\mathcal{P}}}(\R^N,\C)$ is a weak solution of \eqref{prop1.1} if
\[\langle u,\psi\rangle_{A,V_{\mathcal{P}}}-\mathfrak{Re}\int_{\mathbb{R}^N}\left(\frac{1}{|x|^{\alpha}}*|u|^{2_{\alpha}^*})\right)|u|^{2_{\alpha}^*-2}u\bar{\psi}\,\dd x-\lambda\; \mathfrak{Re}\int_{\mathbb{R}^N}\left(\frac{1}{|x|^{\alpha}}*|u|^p)\right)|u|^{p-2}u\bar{\psi}\,\dd x=0\]
for all $\psi\in H^1_{A,V_{\mathcal{P}}}(\R^N,\C)$.
\end{definition}
	
Since the derivative of the energy functional $J_{A,V_{\mathcal{P}}}$ is given by
\begin{align*}
J'_{A,V_{\mathcal{P}}}(u)\cdot \psi
&=\langle u,\psi\rangle_{A,V_{\mathcal{P}}}-\mathfrak{Re}\int_{\mathbb{R}^N}\bigg(\frac{1}{|x|^{\alpha}}*|u|^{2_{\alpha}^*})\bigg)|u|^{2_{\alpha}^*-2}u\bar{\psi}\,\dd x-\lambda\; \mathfrak{Re}\int_{\mathbb{R}^N}\bigg(\frac{1}{|x|^{\alpha}}*|u|^p)\bigg)|u|^{p-2}u\bar{\psi}\,\dd x,
\end{align*}
we see that critical points of $J_{A,V_{\mathcal{P}}}$ are weak solutions of \eqref{prop1.1}.
	
Note that, if $\psi=u$  we obtain
\begin{equation}\label{j'(u)u}
J'_{A,V_{\mathcal{P}}}(u)\cdot u:=\Vert u\Vert_{A,V_{\mathcal{P}}}^{2}-D(u)-\lambda B(u).
\end{equation}
	
\begin{lemma}\label{gpm}
The functional $J_{A,V_{\mathcal{P}}}$ satisfies the mountain pass geometry. Precisely,
\begin{enumerate}
\item [$(i)$] there exist $\rho,\delta>0$ such that $J_{A,V_{\mathcal{P}}}\big|_S\geq \delta>0$ for any $u\in \mathcal{S}$, where
\[\mathcal{S}=\{u\in H^1_{A,V_{\mathcal{P}}}(\mathbb{R}^{N},\mathbb{C})\,:\, \|u\|_{A,V_{\mathcal{P}}}=\rho\};\]
\item [$(ii)$] for any $u_0\in H^1_{A,V_{\mathcal{P}}}(\mathbb{R}^{N},\mathbb{C})\setminus\{0\}$ there exists $\tau\in (0,\infty)$ such that $\|\tau u_0\|_{V_{\mathcal{P}}}>\rho$ and $J_{A,V_{\mathcal{P}}}(\tau u_0) <0$.
\end{enumerate}
\end{lemma}
\begin{proof} Inequalities  \eqref{desigb} and \eqref{desigd} yields
\[J_{A,V_{\mathcal{P}}}(u)\geq \frac{1}{2}\|u\|_{A,V_{\mathcal{P}}}^2-\frac{C_2(\alpha,N)}{2\cdot 2_{\alpha}^*}\|u\|_{A,V_{\mathcal{P}}}^{2\cdot 2_{\alpha}^*}-\frac{\lambda C_1(\alpha,N)}{2p}\|u\|_{A,V_{\mathcal{P}}}^{2p},\]
thus implying ($i$) if we take $\|u\|_{A,V_{\mathcal{P}}}=\rho>0$ sufficiently small.
		
In order to prove ($ii$), fix $u_0\in H^1_{A,V_{\mathcal{P}}}(\mathbb{R}^{N},\mathbb{C})\setminus\{0\}$  and consider the function $g_{u_0}\colon(0,\infty)\to\mathbb{R}$ given by
\[g_{u_0}(t):=J_{A,V_{\mathcal{P}}}(tu_0)=\frac{1}{2} \|tu_0\|^2_{A,V_{\mathcal{P}}}- \frac{1}{2\cdot 2^*_{\alpha}}D(t u_0)-\frac{\lambda}{2p}B(t u_0).\]
		
We have
\begin{align*}
B(tu_0)=
& t^{2p}\int_{\R^N}\int_{\R^N}\frac{|u_0(x)^p| |u_0(y)|^p}{|x-y|^{\alpha}}\dd x\dd y=t^{2p}B(u_0)
\end{align*}
and
\begin{align*}
D(tu_0)=
&t^{2\cdot 2^*_{\alpha}}\int_{\R^N}\int_{\R^N}\frac{| u_0(x)|^{2_{\alpha}^*} |u_0(y)|^{2_{\alpha}^*}}{|x-y|^{\alpha}}\dd x\dd y=t^{2 \cdot 2_{\alpha}^*}D(u_0).
\end{align*}

Thus,
\begin{align*}
g_{u_0} (t)&=\frac{1}{2} t^2\|u_0\|^2_{A,V_{\mathcal{P}}}- \frac{1}{2\cdot 2^*_{\alpha}}t^{2\cdot 2^*_{\alpha}}D( u_0)-\frac{\lambda}{2p}t^{2p}B( u_0)\\
&=\frac{1}{2} t^{2\cdot2^*_{\alpha}}\left(\frac{\|u_0\|^2_{A,V_{\mathcal{P}}}}{t^{(2( 2^*_{\alpha}-1))}}- \frac{1}{2^*_{\alpha}} D( u_0)-\frac{\lambda}{p}\frac{B( u_0)}{t^{(2(2^*_{\alpha}-p)) }}\right)
\end{align*}

Since $1<p<2^*_{\alpha}$, we have
\[\lim_{t\to +\infty}J_{A,V_{\mathcal{P}}}(t u_0)=-\infty\]
completing the proof of ($ii$). 
$\hfill\Box$\end{proof}\vspace*{.3cm}
	
The mountain pass theorem without the PS condition (see \cite[Theorem 1.15]{Willem}) yields a Palais-Smale sequence $(u_n)\subset H^1_{A,V_{\mathcal{P}}}(\mathbb{R}^{N},\mathbb{C})$ such that
\begin{equation*}J'_{A,V_{\mathcal{P}}}(u_n)\to 0\qquad\textrm{and}\qquad J_{A,V_{\mathcal{P}}}(u_n)\to c_{\lambda},\end{equation*}
where
\begin{equation*}
c_{\lambda}=\inf_{\alpha\in \Gamma}\max_{t\in [0,1]}J_{A,V_{\mathcal{P}}}(\gamma(t)),
\end{equation*}
and $\Gamma=\left\{\gamma\in C^1\left([0,1],H^1_{A,V_{\mathcal{P}}}(\mathbb{R}^{N},\mathbb{C})\right)\,:\,\gamma(0)=0,\, J_{A,V_{\mathcal{P}}}(\gamma(1))<0\right\}$.

\begin{lemma} Suppose that $u_n\rightharpoonup u_0$ in $H^1_{A,V_{\mathcal{P}}}(\R^N,\C)$. Then
\begin{equation}\label{appliclieb}
\frac{1}{|x|^{\alpha}}*|u_n|^{s}\rightharpoonup\frac{1}{|x|^{\alpha}}*|u_0|^{s}\;\; \textrm{in}\;\;L^{\frac{2N}{\alpha}}(\R^N),
\end{equation}
for all $\frac{2N-\alpha}{N}\leq s\leq2_{\alpha}^*$.
\end{lemma}
\begin{proof}
In this proof we adapt some ideas of \cite{AlvesMinbo2}. We can suppose that $|u_n(x)|^s \to |u_0(x)|^s$ a.e. in $\mathbb{R}^N$ and, as consequence of the immersion \eqref{immersion}, $|u_n|^s$ is bounded in $L^{\frac{2N}{2N-\alpha}}(\mathbb{R}^N)$. Thus, Lemma \ref{lK} allows us to conclude that \[|u_n(x)|^s\rightharpoonup |u_0(x)|^s\;\;\textrm{in}\;L^{\frac{2N}{2N-\alpha}}(\R^N,\C).\]
	
The Hardy-Littlewood-Sobolev inequality allows us to conclude that
\[\frac{1}{|x|^\alpha}*w(x)\in L^{\frac{2N}{\alpha}}(\mathbb{R}^N)\]
for all $w\in L^{\frac{2N}{2N-\alpha}}(\mathbb{R}^N)$. So, we have a continuous linear map from $L^{\frac{2N}{2N-\alpha}}(\mathbb{R}^N)$ to $L^{\frac{2N}{\alpha}}(\mathbb{R}^N)$. A new application of Lemma \ref{lK} yields \eqref{appliclieb}.
$\hfill\Box$\end{proof}

\begin{corollary}\label{lemmaconvderiv}Suppose that $u_n\rightharpoonup u_0$ and consider 	
\[B'(u_n)\cdot \psi=\mathfrak{Re}\int_{\mathbb{R}^N}\bigg(\frac{1}{|x|^{\alpha}}*|u_n|^p)\bigg)|u_n|^{p-2}u_n\bar{\psi}\,\dd x\]
and
\[D'(u_n)\cdot\psi=\mathfrak{Re}\int_{\mathbb{R}^N}\bigg(\frac{1}{|x|^{\alpha}}*|u_n|^{2_{\alpha}^*})\bigg)|u_n|^{2_{\alpha}^*-2}u_n\bar{\psi}\,\dd x,\]
for $\psi\in C^\infty_c(\mathbb{R}^{N},\mathbb{C})$. Then $B'(u_n)\cdot \psi\to B'(u_0)\cdot \psi$ and  $D'(u_n)\cdot \psi\to D'(u_0)\cdot \psi$.
\end{corollary}
\begin{proof}
The immersion \eqref{immersion} guarantees that $|u_n|^{p - 2}u_n$ is bounded in $ L^{\frac{2N}{N+2-\alpha}}(\R^N, \C)$. Since we can suppose that $|u_n(x)|^p \to |u_0(x)|^p$ a.e. in $\mathbb{R}^N$, by applying Lemma \ref{lK}, we conclude that
\begin{equation}\label{appliclieb1}
|u_n|^{p- 2}u_n\rightharpoonup |u_0|^{p - 2}u\quad\;\textrm{in}\;\; L^{\frac{2N}{N+2-\alpha}}(\R^N, \C)
\end{equation}
for all $\frac{2N-\alpha}{N}\leq p\leq2_{\alpha}^*$, as $n\to+\infty$.
	
Combining \eqref{appliclieb} with \eqref{appliclieb1}  yields
\[\left(\frac{1}{|x|^{\alpha}}*|u_n|^{p}\right)|u_n|^{p-2}u_n\rightharpoonup \left(\frac{1}{|x|^{\alpha}}*|u_0|^{p}\right)|u_0|^{p-2}u_0\;\; \textrm{in}\;\;L^{\frac{2N}{N+2}}(\R^N)\]
as $n\to+\infty$, for all $\frac{2N-\alpha}{N}\leq p\leq2_{\alpha}^*$. Consequently, for $\psi\in C^\infty_c(\mathbb{R}^{N},\mathbb{C})$, it follows that		
\[\mathfrak{Re}\int_{\mathbb{R}^N}\bigg(\frac{1}{|x|^{\alpha}}*|u_n|^p\bigg)|u_n|^{p-2}u_n\bar{\psi}\; \dd x \to \mathfrak{Re} \int_{\mathbb{R}^N}\bigg(\frac{1}{|x|^{\alpha}}*|u_0|^p\bigg)|u_0|^{p-2}u_0\bar{\psi}\; \dd x\]
and
\[\mathfrak{Re}\int_{\mathbb{R}^N}\bigg(\frac{1}{|x|^{\alpha}}*|u_n|^{2_{\alpha}^*}\bigg)|u_n|^{2_{\alpha}^*-2}u_n\bar{\psi}\; \dd x \to \mathfrak{Re} \int_{\mathbb{R}^N}\bigg(\frac{1}{|x|^{\alpha}}*|u_0|^{2_{\alpha}^*}\bigg)|u_0|^{2_{\alpha}^*-2}u_0\bar{\psi}\; \dd x,\]
that is,
\[B'(u_n)\cdot \psi\to B'(u_0)\cdot \psi\qquad\text{and}\qquad D'(u_n)\cdot \psi\to D'(u_0)\cdot \psi.\]
$\hfill\Box$\end{proof}

\begin{lemma}\label{boundedsolut1}
If $(u_n)\subset H^1_{A,V_{\mathcal{P}}}(\mathbb{R}^{N},\mathbb{C})$ is a $(PS)_{\lambda}$ sequence for $J_{A,V_{\mathcal{P}}}$, then $(u_n)$ is bounded. In addition, if $u_n\rightharpoonup u$ weakly in $H^1_{A,V_{\mathcal{P}}}(\mathbb{R}^{N},\mathbb{C})$ as $n\to \infty$, then $u$ is ground state solution for problem \eqref{prop1.1}.
\end{lemma}
\begin{proof}
Standard arguments prove that $(u_n)$ is bounded in $H^1_{A,V_{\mathcal{P}}}(\mathbb{R}^{N},\mathbb{C})$. Then, up to a subsequence, we have $u_n\rightharpoonup u$ weakly in $H^1_{A,V_{\mathcal{P}}}(\mathbb{R}^{N},\mathbb{C})$ as $n\to \infty$. 	

From Corollary \ref{lemmaconvderiv} it follows that, for all $\psi \in H^1_{A,V_{\mathcal{P}}}(\mathbb{R}^{N},\mathbb{C})$, we have
\[\mathfrak{Re}\int_{\mathbb{R}^N}\bigg(\frac{1}{|x|^{\alpha}}*|u_n|^p\bigg)|u|^{p-2}u_n\bar{\psi}\; \dd x = \mathfrak{Re} \int_{\mathbb{R}^N}\bigg(\frac{1}{|x|^{\alpha}}*|u|^p\bigg)|u|^{p-2}u\bar{\psi}\; \dd x+o_n(1),\ \textrm{ as }\ n\to\infty,
\]
where $s=p$ or $s=2_{\alpha}^*$.

Thus, since for all $\psi\in C^\infty_c(\mathbb{R}^{N},\mathbb{C})$ we have $J'_{A,V_{\mathcal{P}}}(u_n)\cdot \psi= o_n (1)$, 
we obtain
\[J'_{A,V}(u)\cdot \psi=0,\;\;\forall\;\psi \in H^1_{A,V_{\mathcal{P}}}(\mathbb{R}^{N},\mathbb{C}),
\]
that is, $u$ is a ground state solution for \eqref{prop1.1}.
$\hfill\Box$\end{proof}\vspace*{.2cm}

We now consider the Nehari manifold associated with the
$J_{A,V_{\mathcal{P}}}$.
\begin{align*}\mathcal{M}_{A,V_{\mathcal{P}}}
&=\left\{u\in H^1_{A,V_{\mathcal{P}}}(\mathbb{R}^{N},\mathbb{C})\setminus\{0\}\,:\,\Vert u\Vert_{A,V_{\mathcal{P}}}^{2}=D(u)+\lambda B(u)\right\}.
\end{align*}
\begin{lemma}\label{equality}
There exists a unique $t_u=t_u (u)>0$  such that $t_u u\in \mathcal{M}_{A,V_{\mathcal{P}}} $ for all  $u\in H^1_{A,V_{\mathcal{P}}}(\mathbb{R}^{N},\mathbb{C})\setminus\{0\}$  and  $J_{A,V_{\mathcal{P}}}(t_u u)=\displaystyle \max_{t\geq 0} J_{A,V_{\mathcal{P}}}(tu)$. Moreover $c_{\lambda}=c^*_{\lambda}=c^{**}_{\lambda}$, where \[c_{\lambda}^{*}= \displaystyle  \inf_{u \;\in\; \mathcal{M}_{A,V}} J_{A,V_{\mathcal{P}}}(u)\quad \textrm{and}\quad c_{\lambda}^{**}=\displaystyle \inf_{u\;\in\; H^1_{A,V_{\mathcal{P}}}(\R^N,\C)\setminus \{0\}}\max_{t\geq 0} J_{A,V_{\mathcal{P}}}(tu).\]
\end{lemma}
\begin{proof} Let $u\in H^1_{A,V_{\mathcal{P}}}(\mathbb{R}^{N},\mathbb{C})\setminus\{0\}$ and $g_u$ defined on $(0,+\infty)$ given by \[g_u(t)=J_{A,V_{\mathcal{P}}}(tu).\] By  the mountain pass geometry (Lemma \ref{gpm}), there exists $t_u>0$ such that
\[g_u(t_u)=\displaystyle \max_{t\geq 0} g_u(t)= \displaystyle \max_{t\geq 0} J_{A,V_{\mathcal{P}}}(t_u u).\]

Hence
\[0=g'_u(t_u)=J'_{A,V_{\mathcal{P}}}(t_u u)\cdot u= J'_{A,V_{\mathcal{P}}}(t_u u)\cdot t_u u,\]
implying that $t_u u \in \mathcal{M}_{A,V_{\mathcal{P}}}$, as consequence of \eqref{j'(u)u}. We now show that $t_u$ is unique. To this end, we suppose that there exists $s_u>0$ such that $s_u u \in \mathcal{M}_{A,V_{\mathcal{P}}}$. Thus, we have both
\begin{equation*}
\|u\|_{A,V_{\mathcal{P}}}^2=t_u^{2(2_{\alpha}^*-1)}D(u)+\lambda t_u^{2(p-1)}B(u)\qquad\text{and}\qquad \|u\|_{A,V_{\mathcal{P}}}^2=s_u^{2(2_{\alpha}^*-1)}D(u)+\lambda s_u^{2(p-1)}B(u).
\end{equation*}

Hence
\[0=\left(t_u^{2^(2_{\alpha}^*-1)}-s_u^{2(2_{\alpha}^*-1)}\right)D(u)+\lambda\left(t_u^{2(p-1)}-s_u^{2(p-1)}\right) B(u).\]

Since both terms in parentheses have the same sign if $t_u\neq s_u$ and we also have $B(u)>0$, $D(u)>0$ and $\lambda>0$, it follows that $t_u=s_u$.
		
Now, the rest of the proof follows arguments similar to that found in \cite{Claudionor, Felmer, Rabinowitz,Willem}.
$\hfill\Box$\end{proof}
	
The following result controls the level ${c_{\lambda}}$ of a Palais-Smale sequence of $J_{A,V_{\mathcal{P}}}$.
		
\begin{lemma}\label{conseqlions}
Let $(u_n)\subset H^1_{A,V_{\mathcal{P}}}(\R^N,\C)$  a $(PS)_{c_{\lambda}}$ sequence for $J_{A,V_{\mathcal{P}}}$ such that
\begin{equation*}
u_n \rightharpoonup 0\quad \textrm{weakly in}\; H^1_{A,V_{\mathcal{P}}}(\R^N,\C), \;\textrm{ as}\; n\to\infty,
\end{equation*}
with
\begin{equation*} c_{\lambda}<\frac{N+2-\alpha}{2(2N- \alpha)}S_{A}^{\frac{2N-\alpha}{N-\alpha +2}}.
\end{equation*}
Then the sequence $(u_n)$ verifies either
\begin{enumerate}
	\item [$(i)$] $u_n\to 0$ strongly in $H^1_{A,V_{\mathcal{P}}}(\R^N,\C),$ as $n\to\infty,$
\end{enumerate}
or
\begin{enumerate}
	\item [$(ii)$] There exists a sequence $(y_n)\subset \R^N$ and constants $r,\theta>0$ such that
	\[\limsup_{n\to\infty}\int_{B_r(y_n)}|u_n|^2 \; \dd x\geq \theta\]
\end{enumerate}
where $B_r(y)$ denotes the ball in $\R^N$ of center at $y$ and  radius $r>0$.		
	\end{lemma}
\begin{proof}
Suppose that ($ii$) does not hold. Applying a result by Lions \cite[Lemma 1.21]{Willem}, it follows from  inequality \eqref{desigb} that
\begin{equation}\label{converb}
B(u_n)\to 0,\quad \textrm{as}\quad n\to \infty.
\end{equation}

Since $J'_{A,V_{\mathcal{P}}}(u_n)u_n=o_n(1)$ as $n\to \infty$, we obtain
\begin{equation}\label{on1}
\|u_n\|^2_{A,V_{\mathcal{P}}}=D(u_n)+o_n (1)\;\;\textrm{as}\;\; n\to \infty.
\end{equation}

Let us suppose that 	
\begin{equation*}
\|u_n\|^2_{A,V_{\mathcal{P}}}\to \ell\ \ (\ell>0)\quad \textrm{as}\quad n\to \infty.
\end{equation*}
Thus, as consequence of \eqref{on1}, we have
\begin{equation*}
D(u_n)\to \ell,\quad\textrm{as}\quad n\to \infty.
\end{equation*}
	
Since
\[J_{A,V}(u_n)=\frac{1}{2}\Vert u\Vert^2_{A,V}-\frac{\lambda}{2p}B(u_n)-\frac{1}{2\cdot 2_{\alpha}^*}D(u_n),\]
making $n\to \infty$ yields 
\begin{align}\label{clambda}
c_{\lambda}=\frac{\ell}{2}\left(1-\frac{1}{ 2_{\alpha}^*}\right)=\ell\left(\frac{N+2-\alpha}{ 2(2N-\alpha)}\right).
\end{align}

On the other hand, it follows from \eqref{dunshl} that
\begin{equation*}
\|u_n\|^2_{A,V_{\mathcal{P}}}\geq \displaystyle \int_{\R^N} |\nabla_A u_n|^2\; \dd x\geq S_{A}(D(u_n))^{\frac{N-2}{2N-\alpha}}, \quad \forall\; u\in D^{1,2}_A(\R^N).
\end{equation*}
Thus, 
\begin{equation}\label{clambda1}
\ell\geq (S_{A})^{\frac{2N-\alpha}{N+2-\alpha}}
\end{equation}
and from  \eqref{clambda} and \eqref{clambda1} we conclude that $c_{\lambda}\geq \frac{N+2-\alpha}{2(2N- \alpha)}S_{A}^{\frac{2N-\alpha}{N+2-\alpha}}$, which is a contradiction.
	Therefore, ($i$) is valid and the proof is complete.
\end{proof}

We now state our result about the periodic problem \eqref{prop1.1}.
\begin{thm}\label{teoprop1}
Under the hypotheses already stated on $A$ and $\alpha$, suppose that $(V_1)$ is valid. Then problem \eqref{prop1.1} has at least one ground state solution if either
\begin{enumerate}
\item [$(i)$] $\frac{N+2-\alpha}{N-2}<p<2_{\alpha}^*$, $N=3,4$ and $\lambda>0$;
\item [$(ii)$] $\frac{2N-\alpha}{N}<p\leq \frac{N+2-\alpha}{N-2}$,  $N=3,4$ and $\lambda$ sufficiently large;
\item [$(iii)$] $\frac{2N-\alpha-2}{N-2}<p<2_{\alpha}^*$, $N\geq 5$ and $\lambda>0$;
\item [$(iv)$] $\frac{2N-\alpha}{N}<p\leq \frac{2N-\alpha-2}{N-2}$,  $N\geq 5$ and $\lambda$ sufficiently large.
\end{enumerate}
\end{thm}

\begin{proof} Let $c_\lambda$ be the mountain pass level and consider a sequence $(u_n)\subset H^1_{A,V_{\mathcal{P}}}(\mathbb{R}^{N},\mathbb{C})$ such that
\[J'_{A,V_{\mathcal{P}}}(u_n)\to 0\qquad\textrm{and}\qquad J_{A,V_{\mathcal{P}}}(u_n)\to c_{\lambda}.\]

\emph{Claim.} We affirm that $c_{\lambda}< \frac{N+2-\alpha}{2(2N-\alpha)}(S_{A})^{\frac{2N-\alpha}{N+2-\alpha}}$, a result that will be shown after completing our proof, since it is very technical.
	
Lemma \ref{boundedsolut1} guarantees that $(u_n)$ is bounded. So, passing to a subsequence if necessary, there is $u\in H^1_{A,V_{\mathcal{P}}}(\mathbb{R}^{N},\mathbb{C})$ such that
\[u_n\rightharpoonup u\ \ \textrm{in}\ \ H^1_{A,V_{\mathcal{P}}}(\R^N,\C), \qquad u_n\to u\ \ \textrm{in}\ \ L^2_{loc}(\R^N,\C)\qquad\text{and}\qquad u_n\rightarrow  u\ \ \textrm{a.e.}\ \ x\;\in \R^N.\]
	
If $u\neq 0$ we are done. If $u=0$, it follows from Lemma \ref{conseqlions} the existence of $\theta>0$ and $(y_n)\subset \R^N$ such that
\begin{equation}\label{conseqlim}
\limsup_{n\to\infty}\int_{B_r(y_n)}|u_n|^2 \; \dd x\geq \theta.
\end{equation}

A direct computation shows that we can assume that $(y_n)\subset \Z^N$. Let
\[v_n(x):=u_n (x+y_n).\]
	
Since both $V_{\mathcal{P}}$ and $A$ are $\Z^N$-periodic, we have
\[\|v_n\|_{A,V_{\mathcal{P}}}=\|u_n\|_{A,V_{\mathcal{P}}}\quad J_{A,V_{\mathcal{P}}}(v_n)=  J_{A,V_{\mathcal{P}}}(u_n)\quad\textrm{and}\quad J'_{A,V_{\mathcal{P}}}(v_n)\to 0,\ \ \textrm{as}\ \ n\to\infty.\]
Therefore there exists  $v\in H^1_{A,V_{\mathcal{P}}}$ such that $v_n \rightharpoonup v$ weakly in $H^1_{A,V_{\mathcal{P}}}(\R^N,\C)$ and $v_n\to v$ in $L^2_{loc}(\R^N,\C)$.
	
We claim that $v\neq 0$. In fact, it follows from \eqref{conseqlim}
\begin{align*}
0<\theta\leq \|v_n\|_{L^2 (B_r(0))}\leq \|v_n-v\|_{L^2 (B_r(0))}+\|v\|_{L^2 (B_r(0))}.
\end{align*}
Since $v_n\rightarrow  v$ in $L^2_{loc}(\R^N)$, we have $\|v_n-v\|_{L^2 (B_r(0))}\to 0$ as $n\to \infty$, proving our claim.
	
But Corollary \ref{lemmaconvderiv} guarantees that $J'_{A,V_{\mathcal{P}}}(v_n)\cdot \psi\to J'_{A,V_{\mathcal{P}}}(v_n)\cdot \psi$ and it follows that $J'_{A,V_{\mathcal{P}}}(v)\cdot \psi=0$. Consequently, $v$ is a ground state solution of problem \eqref{prop1.1}.
$\hfill\Box$\end{proof}\vspace*{.2cm}

We now prove the postponed Claim, that is, we show that $c_{\lambda}< \frac{N+2-\alpha}{2(2N-\alpha)}(S_{A})^{\frac{2N-\alpha}{N+2-\alpha}}$. Observe that, once proved the existence of $u_\epsilon$ as in our next result, then
\[0<c_{\lambda}=\inf_{\alpha\in \Gamma}\max_{t\in [0,1]}J_{A,V_{\mathcal{P}}}(\gamma(t))\leq\sup_{t\geq 0}J_{A,V_{\mathcal{P}}}(tu_{\varepsilon})<\frac{N+2-\alpha}{2(2N-\alpha)}(S_{A})^{\frac{2N-\alpha}{N+2-\alpha}}.\]
\begin{lemma}\label{mainlemma}
There exists $u_{\varepsilon}$ such that
\begin{equation}
\sup_{t\geq 0}J_{A,V_{\mathcal{P}}}(tu_{\varepsilon})<\frac{N+2-\alpha}{2(2N-\alpha)}(S_{A})^{\frac{2N-\alpha}{N+2-\alpha}}.
\end{equation}
provided that either
\begin{enumerate}
\item [$(i)$] $\frac{N+2-\alpha}{N-2}<p<2_{\alpha}^*$, $N=3,4$ and $\lambda>0$;
\item [$(ii)$] $\frac{2N-\alpha}{N}<p\leq \frac{N+2-\alpha}{N-2}$,  $N=3,4$ and $\lambda$ sufficiently large;
\item [$(iii)$] $\frac{2N-2-\alpha}{N-2}<p<2_{\alpha}^*$, $N\geq 5$ and $\lambda>0$;
\item [$(iv)$] $\frac{2N-\alpha}{N}<p\leq \frac{2N-2-\alpha}{N-2}$,  $N\geq 5$ and $\lambda$ sufficiently large.
\end{enumerate}
\end{lemma}
	
The arguments of this proof were adapted from the articles \cite{BrezisN1,Olimpio}. Observe that the conditions stated in this result are exactly the same of Theorem \ref{teoprob1} and Theorem \ref{teoprop1}.
	
\begin{proof} We know that  $U(x)=\frac{[N(N-2)]^{\frac{N-2}{4}}}{(1+|x|^2)^{\frac{N-2}{2}}}$ is a minimizer for $S$, the best Sobolev constant of the immersion $D^{1,2}(\R^N)\hookrightarrow L^{2^*}(\R^N)$ (see \cite[Theorem 1.42]{Willem} or \cite[Section 3]{ArioliSzulkin}) and also a minimizer for $S_{H,L}$, according to Proposition \ref{propsobol}.
	
If $B_r$ denotes the ball in $\R^N$ of center at origin and radius $r$, consider the balls  $B_{\delta}$ and $B_{2\delta}$ and take $\psi \in C_0^{\infty}(\R^N)$ such that, for a constant $C>0$,
\[\psi(x) =  \left\{
\begin{array}{ll}
1,  &  \textrm{if}\;x \in B_{\delta},\\
0,  &  \textrm{if}\;x \in \R^N\setminus B_{2\delta},\end{array}\right. \qquad
0\leq|\psi(x)|\leq 1,\ \ |D\psi(x)|\leq C,\quad \forall\;x\in\R^N.
\]

We define, for $\varepsilon>0,$
\begin{equation}\label{defU}
U_\varepsilon(x):=\varepsilon^{(2-N)/2}U\left(\displaystyle\frac{x}{\varepsilon}\right)\quad\text{and}\qquad u_\varepsilon (x):=\psi(x)U_\varepsilon(x)
\end{equation}

In the proof we apply the estimates 
\begin{equation}\label{estimgrad}
\displaystyle \int_{\R^N}|\nabla u_{\varepsilon}|^2 \dd x =C(N,\alpha)^{\frac{N-2}{2N-\alpha}\cdot \frac{N}{2}}S_{A}^{\frac{N}{2}}+O(\varepsilon^{N-2})
\end{equation}
and
\begin{equation}\label{estimd}
\int_{\R^N}\int_{\R^N}\frac{|u_{\varepsilon}(x)|^{2_{\alpha}^*} |u_{\varepsilon}(y)|^{2_{\alpha}^*}}{|x-y|^{\alpha}}\dd x\dd y\geq C(N,\alpha)^{\frac{N}{2}}S_{A}^{\frac{2N-\alpha}{2}}-O(\varepsilon^{N-\frac{\alpha}{2}}),
\end{equation}
which were obtained by Gao and Yang \cite{Minbo1}.
\vspace{0.5 cm}

\textbf{Case 1.} $\frac{N+2-\alpha}{N-2}<p<2_{\alpha}^*$ and  $N=3,4$  or $\frac{2N-2-\alpha}{N-2}<p<2_{\alpha}^*$ and $N\geq 5$.

\textit{Proof of Case 1.} Consider the function $f:[0,+\infty)\to \R$ defined by
\[f(t)=J_{A,V_{\mathcal{P}}}(tu_{\varepsilon})=\frac{t^2}{2}\|u_{\varepsilon}\|^2_{A,V_{\mathcal{P}}}-\frac{t^{2\cdot 2_{\alpha}^*}}{2\cdot 2_{\alpha}^*}D(u_{\varepsilon})-\frac{\lambda t^{2p}}{2p}B(u_{\varepsilon}).\]
 
The mountain pass geometry (Lemma \ref{gpm}) implies the existence of $t_{\varepsilon}>0$ such that  $\displaystyle \sup_{t\geq 0} J_{A,V_{\mathcal{P}}} (t u_{\varepsilon})=J_{A,V_{\mathcal{P}}} (t_{\varepsilon} u_{\varepsilon})$.
Since $t_{\varepsilon}>0$, $B(u_{\varepsilon})>0$ and  $f'(t_{\varepsilon})=0$, we obtain
\[0<t_{\varepsilon}<\left(\frac{\|u_{\varepsilon}\|_{A,V_{\mathcal{P}}}^2}{D(u_{\varepsilon})}\right)^{\frac{1}{2(2_{\alpha}^*-1)}}:=S_{A}(\varepsilon),\]
thus implying 
\begin{equation}\label{SHepsilon}
\|u_{\varepsilon}\|_{A,V_{\mathcal{P}}}^2= D(u_{\varepsilon})\left(S_{A}(\varepsilon)\right)^{2(2_{\alpha}^*-1)}.
\end{equation}

Now define $g:[0,S_{A}(\varepsilon)]\to\R$ by
\[g(t)=\frac{t^2}{2}\|u_{\varepsilon}\|^2_{A,V_{\mathcal{P}}}-\frac{t^{2\cdot2_{\alpha}^* }}{2\cdot 2_{\alpha}^*}D(u_\varepsilon).\]
So, 
\[g(t)=\frac{t^2}{2}D(u_{\varepsilon})\left(S_{A}(\varepsilon)\right)^{2(2_{\alpha}^*-1)}-\frac{t^{2\cdot2_{\alpha}^* }}{2\cdot 2_{\alpha}^*}D(u_{\varepsilon}).\]

Since $t>0$ and $D(u_{\varepsilon})>0$, it follows that  $g'(t)>0$,  and, consequently, $g$ is increasing in this interval. Thus,
\begin{equation*}
0<g(t_{\varepsilon})<\frac{N+2-\alpha}{2(2N-\alpha)}D(u_{\varepsilon})(S_{A}(\varepsilon))^{2\cdot 2_{\alpha}^*}.
\end{equation*}

We conclude that
\begin{equation*}
D(u_{\varepsilon})(S_{A}(\varepsilon))^{2\cdot 2_{\alpha}^*}=\frac{(\|u_{\varepsilon}\|_{A,V_{\mathcal{P}}}^2)^{\frac{2N-\alpha}{N+2-\alpha}}}{D(u_{\varepsilon})^{\frac{N-2}{N+2-\alpha}}}
\end{equation*}
and therefore
\begin{equation*}
0<g(t_{\varepsilon})<\frac{N+2-\alpha}{2(2N-\alpha)}\cdot\frac{(\|u_{\varepsilon}\|_{A,V_{\mathcal{P}}}^2)^{\frac{2N-\alpha}{N+2-\alpha}}}{D(u_{\varepsilon})^{\frac{N-2}{N+2-\alpha}}}.
\end{equation*}

Since $J_{A,V_{\mathcal{P}}}(t u_{\varepsilon})=g(t)-\frac{\lambda}{2p}t^{2p}B(u_{\varepsilon})$, we have
\begin{equation*}
J_{A,V_{\mathcal{P}}}(t_{\varepsilon} u_{\varepsilon})< \frac{N+2-\alpha}{2(2N-\alpha)}\left(\frac{\|u_{\varepsilon}\|_{A,V_{\mathcal{P}}}^2}{D(u_{\varepsilon})^{\frac{N-2}{2N-\alpha}}}\right)^{\frac{2N-\alpha}{N+2-\alpha}}-\frac{\lambda}{2p}t_{\varepsilon}^{2p}B(u_{\varepsilon}).
\end{equation*}

But $\|u_{\varepsilon}\|_{A,V_{\mathcal{P}}}^2=\int_{\R^N}|\nabla u_{\varepsilon}|^2 \dd x+\int_{\R^N}(|A(x)|^2+V_{\mathcal{P}}(x)|u_{\varepsilon}|^2) \dd x$ implies
\begin{align*}
\frac{\|u_{\varepsilon}\|_{A,V_{\mathcal{P}}}^2}{D(u_{\varepsilon})^{\frac{N-2}{2N-\alpha}}}
&=\frac{1}{(D(u_{\varepsilon}))^{\frac{N-2}{2N-\alpha}}}\int_{\R^N}|\nabla u_{\varepsilon}|^2 \dd x +\frac{1}{(D(u_{\varepsilon}))^{\frac{N-2}{2N-\alpha}}}\int_{\R^N}(|A(x)|^2+V_{\mathcal{P}}(x)|u_{\varepsilon}|^2 )\dd x.
\end{align*}	
Therefore, we conclude that	
\begin{align*}
J_{A,V_{\mathcal{P}}}(t_{\varepsilon}u_{\varepsilon})
&<\frac{N+2-\alpha}{2(2N-\alpha)}\left(\frac{1}{(D(u_{\varepsilon}))^{\frac{N-2}{2N-\alpha}}}\displaystyle \int_{\R^N}|\nabla u_{\varepsilon}|^2 \dd x \right.\\
&\quad\left.+\frac{1}{(D(u_{\varepsilon}))^{\frac{N-2}{2N-\alpha}}}\displaystyle\int_{\R^N}(|A(x)|^2+V_{\mathcal{P}}(x))|u_{\varepsilon}|^2 \dd x\right)^{\frac{2N-\alpha}{N+2-\alpha}}-\frac{\lambda}{2p}t_{\varepsilon}^{2p}B(u_{\varepsilon}).
\end{align*}

Since, for all   $\beta\geq1$ and any $a,b>0$ we have $(a+b)^{\beta}\leq a^{\beta}+\beta(a+b)^{\beta-1} b$,
considering
\[a=\frac{1}{(D(u_{\varepsilon}))^{\frac{N-2}{2N-\alpha}}}\displaystyle\int_{\R^N}|\nabla u_{\varepsilon}|^2 \dd x,\quad  b=\frac{1}{(D(u_{\varepsilon}))^{\frac{N-2}{2N-\alpha}}}\displaystyle \int_{\R^N}(|A(x)|^2+V_{\mathcal{P}}(x)|u_{\varepsilon}|^2 )\dd x\quad\text{and}\quad \beta=\frac{2N-\alpha}{N+2-\alpha},\]
it follows 
\begin{align}\label{SHepsilon6}
J_{A,V_{\mathcal{P}}}(t_{\varepsilon}u_\varepsilon)
&<\frac{N+2-\alpha}{2(2N-\alpha)}\left[\left(\frac{1}{(D(u_{\varepsilon}))^{\frac{N-2}{2N-\alpha}}}\displaystyle\int_{\R^N}|\nabla u_{\varepsilon}|^2 \dd x \right)^{\frac{2N-\alpha}{N+2-\alpha}}\right.\\ &\quad+\frac{2N-\alpha}{N+2-\alpha}\left(\frac{1}{D(u_{\varepsilon})^{\frac{N-2}{2N-\alpha}}}\displaystyle \int_{\R^N} |\nabla u_{\varepsilon}|^2 \dd x + \frac{1}{(D(u_{\varepsilon})^{\frac{N-2}{2N-\alpha}}}\displaystyle \int_{\R^N}(|A(x)|^2+V_{\mathcal{P}}(x)|u_{\varepsilon}|^2 )\dd x\right)^{\frac{N-2}{N+2-\alpha}}\nonumber\\
&\quad \cdot\left. \frac{1}{((D(u_{\varepsilon}))^{\frac{N-2}{2N-\alpha}}}\displaystyle \int_{\R^N}(|A(x)|^2+V_{\mathcal{P}}(x)|u_{\varepsilon}|^2 )\dd x\right]-\frac{\lambda}{2p}t_{\varepsilon}^{2p}B(u_{\varepsilon})\nonumber.
\end{align}

Taking into account \eqref{estimgrad} and \eqref{estimd}, we conclude that
\begin{align}\label{SHepsilon25}
\left(\frac{1}{(D(u_{\varepsilon}))^{\frac{N-2}{2N-\alpha}}}\displaystyle\int_{\R^N}|\nabla u_{\varepsilon}|^2 \dd x\right)^{\frac{2N-\alpha}{N+2-\alpha}}&\leq\left(\frac{(C(N,\alpha))^{\frac{N-2}{2N-\alpha}\cdot \frac{N}{2}}\cdot S_{H,L}^{\frac{N}{2}}+O(\varepsilon^{N-2})}{\left(C(N,\alpha)^{\frac{N}{2}}S_{H,L}^{\frac{2N-\alpha}{2}}-O(\varepsilon^{\frac{2N-\alpha}{2}})\right)^{\frac{N-2}{2N-\alpha}}}\right)^{\frac{2N-\alpha}{N+2-\alpha}}.
	\end{align}

We also have
\begin{equation*}
\left(\frac{(C(N,\alpha))^{\frac{N-2}{2N-\alpha}\cdot \frac{N}{2}}(S_{H,L})^{\frac{N}{2}}+O(\varepsilon^{N-2})}{\left(C(N,\alpha)^{\frac{N}{2}}S_{H,L}^{\frac{2N-\alpha}{2}}- O(\varepsilon^{\frac{2N-\alpha}{2}})\right)^{\frac{N-2}{2N-\alpha}}}\right)^{\frac{2N-\alpha}{N+2-\alpha}}=(S_{H,L})^{\frac{2N-\alpha}{N+2-\alpha}}\cdot\left( \frac{1+O(\varepsilon^{N-2})}{\left(1-O\left(\varepsilon^{\frac{2N-\alpha}{2}}\right)\right)^{\frac{N-2}{2N-\alpha}}}\right)^{\frac{2N-\alpha}{N+2-\alpha}}
\end{equation*}
and
\begin{align*}
\left(\frac{1+O(\varepsilon^{N-2})}{\left(1-O\left(\varepsilon^{\frac{2N-\alpha}{2}}\right)\right)^{\frac{N-2}{2N-\alpha}}}\right)^{\frac{2N-\alpha}{N+2-\alpha}}
< 1+C(N,\alpha)\cdot \frac{O(\varepsilon^{N-2})+O(\varepsilon^{\frac{2N-\alpha}{2}})}{\left(1-O(\varepsilon^{\frac{2N-\alpha}{2}})\right)^{\frac{N-2}{2N-\alpha}}}.
\end{align*}

We observe that, for $\varepsilon>0$ sufficiently small, it holds
\[(1-O(\varepsilon^{\frac{N-2}{2N-\alpha}}))^{\frac{N-2}{2N-\alpha}}\geq \frac{1}{2}.\] So, 	
\begin{align*}
\left(\frac{1+O(\varepsilon^{N-2})}{\left(1-O\left(\varepsilon^{\frac{2N-\alpha}{2}}\right)\right)^{\frac{N-2}{2N-\alpha}}}\right)^{\frac{2N-\alpha}{N+2-\alpha}}<1+2C(N,\alpha)\left(O\left(\varepsilon^{N-2}\right)+O\left(\varepsilon^{\frac{2N-\alpha}{2}}\right)\right)< 1+ O\left(\varepsilon^{\min\{N-2,\frac{2N-\alpha}{2}\}}\right).
\end{align*}	

Therefore, we conclude that, for any $\varepsilon>0$ sufficiently small, we have
\begin{align}\label{SHepsilon10}
\left(\frac{1}{(D(u_{\varepsilon}))^{\frac{N-2}{2N-\alpha}}}\displaystyle\int_{\R^N}|\nabla u_{\varepsilon}|^2 \dd x\right)^{\frac{2N-\alpha}{N+2-\alpha}}
 < \left(S_{H,L}\right)^{\frac{2N-\alpha}{N+2-\alpha}}+ O\left(\varepsilon^{\min\{N-2,\frac{2N-\alpha}{2}\}}\right).
\end{align}	

Combining \eqref{SHepsilon6} with \eqref{SHepsilon10},   for $\varepsilon$ sufficiently small, we have
\begin{align}\label{SHepsilon11}
J_{A,V_{\mathcal{P}}}(t_{\varepsilon}u_{\varepsilon})&<\frac{N+2-\alpha}{2(2N-\alpha)}\left(S_{H,L} \right)^{\frac{2N-\alpha}{N+2-\alpha}}+ O\left(\varepsilon^{\min\{N-2,\frac{2N-\alpha}{2}\}}\right)\\ &\quad+\displaystyle\frac{1}{2}\left(\frac{1}{D(u_{\varepsilon})^{\frac{N-2}{2N-\alpha}}}\displaystyle \int_{\R^N} |\nabla u_{\varepsilon}|^2 \dd x + \frac{1}{(D(u_{\varepsilon})^{\frac{N-2}{2N-\alpha}}}\displaystyle \int_{\R^N}(|A(x)|^2+V_{\mathcal{P}}(x))|u_{\varepsilon}|^2\dd x\right)^{\frac{N-2}{N+2-\alpha}} \nonumber\\
&\quad\,\cdot\frac{1}{\left(D(u_{\varepsilon}\right)^{\frac{N-2}{2N-\alpha}}}\displaystyle \int_{\R^N}(|A(x)|^2+V_{\mathcal{P}}(x))|u_{\varepsilon}|^2 \dd x-\frac{\lambda}{2p}t_{\varepsilon}^{2p}B(u_{\varepsilon}).\nonumber
\end{align}		

We claim that there is a positive constant $C_0$ such that, for all $\varepsilon>0$
\begin{equation}\label{SHepsilon12}
t_{\varepsilon}^{2p}\geq C_0.
\end{equation}

In fact, suppose that there is a sequence $(\varepsilon_n)\subset \R$, $\varepsilon_n\to 0$ as $n\to\infty$, such that $t_{\varepsilon_n}\to 0$ as $n\to \infty$. Thus,
\[0<c_{\lambda}\leq \sup_{t\geq0}J_{A,V}(t u_{\varepsilon_n})=J_{A,V_{\mathcal{P}}}(t_{\varepsilon_n}u_{\varepsilon_n}).\]
Since  $u_{\varepsilon_n}\in H^1_{A,V_{\mathcal{P}}}(\R^N,\C)$ is bounded and $t_{\varepsilon_n}\to 0$, as $n\to\infty$, we have $t_{\varepsilon_n}u_{\varepsilon_n}\to 0$ as $n\to\infty$, em $H^1_{A,V_{\mathcal{P}}}(\R^N,\C)$.

The continuity of $J_{A,V_{\mathcal{P}}}$ implies that $J_{A,V_{\mathcal{P}}}(t_{\varepsilon_n}u_{\varepsilon_n})\to J_{A,V_{\mathcal{P}}}(0)= 0$. Therefore,
\[0<c_{\lambda}\leq \lim_{n\to\infty}J_{A,V_{\mathcal{P}}}(t_{\varepsilon_n}u_{\varepsilon_n})=0,\]
a contradiction that proves the claim.

From \eqref{SHepsilon}, \eqref{SHepsilon11} and \eqref{SHepsilon12}  we conclude that, for some constant $C_0>0$ and $\varepsilon>0$ sufficiently small we have
\begin{align}\label{SHepsilon13}
J_{A,V_{\mathcal{P}}}(t_{\varepsilon}u_{\varepsilon})
& <\frac{N+2-\alpha}{2(2N-\alpha)}\left(S_{A} \right)^{\frac{2N-\alpha}{N+2-\alpha}}+ O\left(\varepsilon^{\min\{N-2,\frac{2N-\alpha}{2}\}}\right)\\ &\quad+\frac{1}{2}\left(\frac{1}{D(u_{\varepsilon})^{\frac{N-2}{2N-\alpha}}}\|u_{\varepsilon}\|^2_{A V_\mathcal{P}}\right)^{\frac{N-2}{N+2-\alpha}}\nonumber\cdot\frac{1}{\left(D(u_{\varepsilon}\right)^{\frac{N-2}{2N-\alpha}}}\displaystyle \int_{\R^N}(|A(x)|^2+V_{\mathcal{P}}(x))|u_{\varepsilon}|^2 \dd x-C_0B(u_{\varepsilon})\nonumber\\
&<\frac{N+2-\alpha}{2(2N-\alpha)}\left(S_{A} \right)^{\frac{2N-\alpha}{N+2-\alpha}}\nonumber\\ 
&\quad+ O\left(\varepsilon^{\min\{N-2,\frac{2N-\alpha}{2}\}}\right)+\frac{S_{A}(\varepsilon)^2}{2}\cdot \displaystyle \int_{\R^N}(|A(x)|^2+V_{\mathcal{P}}(x))|u_{\varepsilon}|^2 \dd x-C_0B(u_{\varepsilon}).\nonumber 
\end{align}

Thus,
\begin{equation}\label{SHepsilon26} J_{A,V_{\mathcal{P}}}(t_{\varepsilon}u_{\varepsilon})<\frac{N+2-\alpha}{2(2N-\alpha)}\left(S_{A} \right)^{\frac{2N-\alpha}{N+2-\alpha}}+ O(\varepsilon^{\eta})+ C_1\displaystyle \int_{\R^N} a(x)|u_{\varepsilon}|^2 \dd x-C_0B(u_{\varepsilon}),
\end{equation}
where $C_1=\frac{S_{A}(\varepsilon)^2}{2}$, $a(x)=|A(x)|^2+V_p(x)$ and $\eta=\min\{N-2,\frac{2N-\alpha}{2}\}$.

By direct computation we know that, for $\varepsilon<1$,
\begin{align*}
B(u_{\varepsilon})&=\displaystyle\int_{\R^N}\displaystyle\int_{\R^N}\frac{\varepsilon^{\frac{(2-N)p}{2}}[N(N-2)]^{\frac{(N-2)p}{4}}\varepsilon^{\frac{(2-N)p}{2}}[N(N-2)]^{\frac{(N-2)p}{4}}}{(1+|\frac{x}{\varepsilon}|^2)^{\frac{(N-2)p}{2}}|x-y|^{\alpha}(1+|\frac{y}{\varepsilon}|^2)^{\frac{(N-2)p}{4}}}\dd x\dd y\\
&=[N(N-2)]^{\frac{(N-2)p}{2}}\varepsilon^{2N-\alpha-(N-2)p}\displaystyle\int_{B_{\frac{\delta}{\varepsilon}}}\displaystyle\int_{B_{B_{\frac{\delta}{\varepsilon}}}}\frac{1}{(1+|x|^2)^{\frac{(N-2)p}{2}}|x-y|^{\alpha}(1+|y|^2)^{\frac{(N-2)p}{2}}}\dd x\dd y\\
&\geq [N(N-2)]^{\frac{(N-2)p}{2}}\varepsilon^{2N-\alpha-(N-2)p}\displaystyle\int_{B_{\delta}}\displaystyle\int_{B_{\delta}}\frac{1}{(1+|x|^2)^{\frac{(N-2)p}{2}}|x-y|^{\alpha}(1+|y|^2)^{\frac{(N-2)p}{2}}}\dd x\dd y.
\end{align*}
Therefore,
\begin{equation*}
B(u_{\varepsilon})\geq [N(N-2)]^{\frac{(N-2)p}{2}}\varepsilon^{2N-\alpha-(N-2)p}\displaystyle\int_{B_{\delta}}\displaystyle\int_{B_{\delta}}\frac{1}{(1+|x|^2)^{\frac{(N-2)p}{2}}|x-y|^{\alpha}(1+|y|^2)^{\frac{(N-2)p}{2}}}\dd x\dd y.
\end{equation*}

Since $a(x)$ is bounded, \eqref{SHepsilon26} and the last inequality imply that
\begin{equation}\label{SHepsilon16}
J_{A,V_{\mathcal{P}}}(t_{\varepsilon}u_{\varepsilon})<
\frac{N+2-\alpha}{2(N-\alpha)}\left(S_{A} \right)^{\frac{2N-\alpha}{N+2-\alpha}}+ O(\varepsilon^{\eta})+ C_2\displaystyle \int_{\R^N}|u_{\varepsilon}(x)|^2 \dd x-C_3 \varepsilon^{2N-\alpha-(N-2)p}.
\end{equation}

We are going to show that
\begin{equation}\label{SHepsilon18}
\lim_{\varepsilon\to 0}\varepsilon^{-\eta}\left(C_2\displaystyle \int_{\R^N}|u_{\varepsilon}(x)|^2 \dd x-C_3\varepsilon^{2N-\alpha-(N-2)p}\right)=-\infty.
\end{equation}

In order to do that, it suffices to show that
\begin{equation}\label{SHepsilon19}
\lim_{\varepsilon\to 0}\varepsilon^{-\eta}\left(C_2\displaystyle \int_{B_{\delta}}|u_{\varepsilon}(x)|^2 \dd x-C_3\varepsilon^{2N-\alpha-(N-2)p}\right)=-\infty
\end{equation}
and
\begin{equation}\label{SHepsilon20}
C_2\displaystyle \int_{B_{2\delta}\setminus B_{\delta}}|u_{\varepsilon}(x)|^2 \dd x-C_3\varepsilon^{2N-\alpha-(N-2)p}=O(\varepsilon^{\eta}).
\end{equation}

Assuming \eqref{SHepsilon19} and \eqref{SHepsilon20}, let us proceed with our proof. Since
\[O(\varepsilon^{\eta})+C_2\displaystyle \int_{\R^N}|u_{\varepsilon}|^2 \dd x-C_3\varepsilon^{2N-\alpha-(N-2)p}=\varepsilon^{\eta}\left[\frac{O(\varepsilon^{\eta})}{\varepsilon^\eta}+\varepsilon^{-\eta}\left(C_2\displaystyle \int_{\R^N}|u_{\varepsilon}(x)|^2 \dd x-C_3\varepsilon^{2N-\alpha-(N-2)p}\right)\right],\]
from \eqref{SHepsilon18} follows
\begin{equation}\label{SHepsilon22}
O(\varepsilon^{\eta})+C_2\displaystyle \int_{\R^N}|u_{\varepsilon}(x)|^2 \dd x-C_3\varepsilon^{2N-\alpha-(N-2)p}<0
\end{equation}
for $\varepsilon>0$ sufficiently small.

Thus, \eqref{SHepsilon16} and \eqref{SHepsilon22} imply
\begin{align*}
\sup_{t\geq0} J_{A,V_{\mathcal{P}}}(t u_{\varepsilon})
<\frac{N+2-\alpha}{2(2N-\alpha)}\left(S_{A} \right)^{\frac{2N-\alpha}{N+2-\alpha}}
\end{align*}
for $\varepsilon>0$ sufficiently small and fixed. Once \eqref{SHepsilon19} and \eqref{SHepsilon20} are verified, the proof of Case 1 is complete.
$\hfill\Box$\end{proof}\vspace*{.4cm}

We now prove \eqref{SHepsilon19}.
\begin{lemma}\label{convergiepsion1} If  $\frac{N+2-\alpha}{N-2}<p<2_{\alpha}^*$ and $N=3,4$ or $\frac{2N-2-\alpha}{N-2}<p<2_{\alpha}^*$ and $N\geq 5$ it follows that
\begin{equation*}
\lim_{\varepsilon\to 0}\varepsilon^{-\eta}\left(C_2\displaystyle \int_{B_{\delta}}|u_{\varepsilon}(x)|^2 \dd x-C_3\varepsilon^{2N-\alpha-(N-2)p}\right)=-\infty
\end{equation*}
\end{lemma}

\begin{proof} This limit is evaluated considering the cases $N=3,$ $N=4$ and $N\geq 5$ as follows. We initially observe that direct computation allows us to conclude that
\begin{equation}\label{SHepsilon17}
\displaystyle \int_{B_{\delta}}|u_{\varepsilon}(x)|^2 \dd x= N\omega_N [N(N-2)]^{\frac{N-2}{2}}\varepsilon^2\displaystyle\int_{0}^{\frac{\delta}{\varepsilon}}\frac{r^{N-1}}{(1+r^2)^{N-2}} \dd r,
\end{equation}
where $\omega_N$ denotes the volume of the unit ball in $\R^N$.

Now, define
\begin{align*}
I_{\varepsilon}:&=\varepsilon^{-\eta}\left(C_2\displaystyle \int_{B_{\delta}}|u_{\varepsilon}(x)|^2 \dd x-C_3\varepsilon^{2N-\alpha-(N-2)p}\right)=\varepsilon^{-\eta}\left(C_4\varepsilon^2\displaystyle\int_{0}^{\frac{\delta}{\varepsilon}}\frac{r^{N-1}}{(1+r^2)^{N-2}} \dd r-C_3\varepsilon^{2N-\alpha-(N-2)p}\right),
\end{align*}
the second equality being a consequence of \eqref{SHepsilon17}.

$\bullet$\ \textbf{The case $\mathbf{N=3}$.} In this case we have $5-\alpha<p<2^*_\alpha$ and therefore $5-\alpha-p<0$. We also observe that $0<\alpha<N$ implies $\min\{N-2,\frac{2N-\alpha}{2}\}=N-2=1$.

It is easy to show that
\begin{equation}\label{SHepsilon28}
\varepsilon^2\displaystyle\int_{0}^{\frac{\delta}{\varepsilon}}\frac{r^2}{1+r^2} \dd r=\varepsilon\left(\delta-\varepsilon\arctan\left(\frac{\delta}{\varepsilon}\right)\right).
\end{equation}
	
Thus,
\begin{align*}
I_{\varepsilon}&=C_4\left(\delta-\varepsilon\arctan\left(\frac{\delta}{\varepsilon}\right)\right)-C_3\varepsilon^{5-\alpha-p}.
\end{align*}
Our claim follows.
		
$\bullet$\ \textbf{The case $\mathbf{N= 4}$.}   In this case, $\frac{6-\alpha}{2}<p<2_{\alpha}^*$ implies $6-\alpha-2p<0$ and $\min\{N-2,\frac{2N-\alpha}{2}\}=N-2=2$, since $0<\alpha<4$.

Changing variables, we obtain
\begin{equation}\label{SHepsilon21}
\varepsilon^2\displaystyle\int_{0}^{\frac{\delta}{\varepsilon}}\frac{r^3}{(1+r^2)^2} \dd r=\frac{\varepsilon^2}{2}\left[\ln \left(1+\frac{\delta^2}{\varepsilon^2}\right) + \frac{\varepsilon^2}{\varepsilon^2+\delta^2}-1\right].
\end{equation}

So, 
\begin{align*}
I_{\varepsilon}&=\frac{C_4}{2} \left(\ln \left(1+\frac{\delta^2}{\varepsilon^2}\right) + \frac{\varepsilon^2}{\varepsilon^2+\delta^2}-1\right)-C_3\varepsilon^{6-\alpha-2p}\\
&=\ln\left(1+\frac{\delta^2}{\varepsilon^2}\right)\left[\frac{C_4}{2}+\frac{C_4}{2\ln\left(1+\frac{\delta^2}{\varepsilon^2}\right)}\frac{\varepsilon^2}{\varepsilon^2+\delta^2}-\frac{C_4}{2\ln\left(1+\frac{\delta^2}{\varepsilon^2}\right)}-C_3\frac{\varepsilon^{6-\alpha-2p}}{\ln\left(1+\frac{\delta^2}{\varepsilon^2}\right)}\right].
\end{align*}
Our claim follows by applying L'Hospital rule.	

\vspace{0.2cm}	

$\bullet$\ \textbf{The case $\mathbf{N\geq 5}$}.  We have
\begin{align*}
I_{\varepsilon}&=\varepsilon^{2-\min\{N-2,\frac{2N-\alpha}{2}\}}\left(C_4\displaystyle\int_{0}^{\frac{\delta}{\varepsilon}}\frac{r^{N-1}}{(1+r^2)^{N-2}} \dd r-C_3\varepsilon^{2N-\alpha-(N-2)p-2}\right).
\end{align*}
It is easy to show that, if $N\geq 5$, then the integral
\[\lim_{\varepsilon\to 0}\int_{0}^{\frac{\delta}{\varepsilon}}\frac{r^{N-1}}{(1+r^2)^{N-2}} \dd r\]
converges.

There are two cases to be considered: 
\begin{itemize} \item $0<\alpha<4$ and  $N\geq 5$;
	\item  $\alpha\geq 4$ and $N\geq 5$.
\end{itemize}

Let us suppose $0<\alpha<4$ and  $N\geq 5$. Since  $0<\alpha<4$  we have
\[2-\eta=2-\min\{N-2, \frac{2N-\alpha}{2}\}=-N+4<0.\]

Also $\frac{2N-\alpha-2}{N-2}<p<\frac{2N-\alpha}{N-2}$ implies $2N-\alpha-(N-2)p-2<0$. Therefore,
$I_{\varepsilon}\to -\infty$ as $\varepsilon\to 0$.

Now we consider the case  $\alpha\geq 4$ and $N\geq 5$. We have $N-2\geq \frac{2N-\alpha}{2}$ and therefore
\[2-\eta=2-\min\bigg\{N-2, \frac{2N-\alpha}{2}\bigg\}=2-N+\frac{\alpha}{2}<0.\]

Since 
\[I_{\varepsilon}=\varepsilon^{2-N+\frac{\alpha}{2}}\left[C_4\int_{0}^{\frac{\delta}{\varepsilon}}\frac{r^{N-1}}{(1+r^2)^{N-2}} \dd r-C_3\varepsilon^{2N-\alpha-(N-2)p-2}\right],\]
we conclude that $I_{\varepsilon}\to -\infty$. We are done.
$\hfill\Box$\end{proof}

We now prove \eqref{SHepsilon20}.
\begin{lemma}\label{SHL}It holds
\begin{equation*}
C_2\displaystyle \int_{B_{2\delta}\setminus B_{\delta}}|u_{\varepsilon}(x)|^2 \dd x-C_3\varepsilon^{2N-\alpha-(N-2)p}=O(\varepsilon^{\eta}).
\end{equation*}
\end{lemma}
\begin{proof}Fix $\delta>0$ sufficiently large so that $U^2_{\varepsilon} (x)\leq \varepsilon^{1+\eta}$ if $|x|\geq \delta$. Since
\begin{align*}
\frac{1}{\varepsilon^{\eta}}\left[C_2\displaystyle \int_{B_{2\delta}\setminus B_{\delta}}|u_{\varepsilon}(x)|^2 \dd x-C_3\varepsilon^{2N-\alpha-(N-2)p}\right]&<\frac{C_2}{\varepsilon^{\eta}}\displaystyle \int_{B_{2\delta}\setminus B_{\delta}} \psi^2 (x) U^2_{\varepsilon}(x)\dd x\leq C_2\varepsilon \|\psi\|_{2}\\
&\leq C_1\varepsilon \|\psi\|_{A,V_{\mathcal{P}}},
\end{align*}
our proof is complete.
$\hfill\Box$\end{proof}
\vspace{0.5 cm}

\textbf{Case 2.} For $\lambda$ sufficiently large, $\frac{2N-\alpha}{N}<p\leq \frac{N+2-\alpha}{N-2}$ and $N=3,4$ or $\frac{2N-\alpha}{N}<p\leq \frac{2N-2-\alpha}{N-2}$ and  $N\geq 5.$ 

\vspace{0.5 cm}

\textit{Proof of Case 2.} Define $g_\lambda:[0,+\infty)\to\R$ by \[g_{\lambda} (t)= J_{A,V_{\mathcal{P}}}(t u_{\varepsilon})= \frac{t^2}{2}\int_{\R^N} \left[|\nabla u_{\varepsilon}|^2+\left(|A(x)|^2+V_{\mathcal{P}}(x)\right)|u_{\varepsilon}|^2\right]\dd x-\frac{\lambda}{2p}t^{2p}B(u_{\varepsilon})-\frac{1}{2\cdot 2_{\alpha}^*}t^{2\cdot 2_{\alpha}^*} D(u_{\varepsilon}).\]
We already know that
$\displaystyle\lim_{t\to +\infty} g_{\lambda}(t)=-\infty$ as $t\to +\infty$ and $\displaystyle \max_{t\geq 0} g_{\lambda} (t) $ is attained at some $t_{\lambda}>0$ satisfying
\[t_{\lambda}\int_{\R^N} \left[|\nabla u_{\varepsilon}|^2+\left(|A(x)|^2+V_{\mathcal{P}}(x)\right)|u_{\varepsilon}|^2\right]\dd x= \lambda t_{\lambda}^{2p-1}B(u_{\varepsilon})+t_{\lambda}^{2\cdot 2_{\alpha}^*-1} D(u_{\varepsilon}),\]
that is,
\[\int_{\R^N} \left[|\nabla u_{\varepsilon}|^2+\left(|A(x)|^2+V_{\mathcal{P}}(x)\right)|u_{\varepsilon}|^2\right]\dd x= \lambda t_{\lambda}^{2(p-1)}B(u_{\varepsilon})+t_{\lambda}^{2( 2_{\alpha}^*-1)} D(u_{\varepsilon}),\]
since $g'_{\lambda} (t_\lambda)=0$. Thus $t_{\lambda}\to 0$ as $\lambda\to+\infty$ and
\begin{align*}
\max_{t\geq 0} J_{A,V_{\mathcal{P}}}(t u_{\varepsilon})
&= \frac{{t_\lambda}^2}{2}\int_{\R^N} \left[|\nabla u_{\varepsilon}(x)|^2+\left(|A(x)|^2+V_{\mathcal{P}}(x)\right)|u_{\varepsilon}(x)|^2\right]\dd x-\frac{\lambda}{2p}{t_\lambda}^{2p}B(u_{\varepsilon})-\frac{1}{2\cdot 2_{\alpha}^*}t^{2\cdot 2_{\alpha}^*} D(u_{\varepsilon})\\
&<  \frac{{t_\lambda}^2}{2}\int_{\R^N} \left[|\nabla u_{\varepsilon}|^2+\left(|A(x)|^2+V_{\mathcal{P}}(x)\right)|u_{\varepsilon}(x)|^2\right]\dd x.
\end{align*}

Since $t_{\lambda}\to 0$ as $\lambda\to +\infty$ and $\frac{N+2-\alpha}{2(N-\alpha)}(S_{A})^{\frac{2N-\alpha}{N+2-\alpha}}>0$, we conclude that
\[\frac{{t_\lambda}^2}{2}\int_{\R^N} \left[|\nabla u_{\varepsilon}|^2+\left(|A(x)|^2+V_{\mathcal{P}}(x)\right)|u_{\varepsilon}(x)|^2\right]\dd x<\frac{N+2-\alpha}{2(2N-\alpha)}(S_{A})^{\frac{2N-\alpha}{N+2-\alpha}},\]
for $\lambda>0$ sufficiently large.

Therefore,
\[\sup_{t\geq 0}J_{A,V_{\mathcal{P}}}(t u_{\varepsilon})<\frac{N+2-\alpha}{2(2N-\alpha)}(S_{A})^{\frac{2N-\alpha}{N+2-\alpha}}\]
for $\lambda>0$ sufficiently large.
$\hfill\Box$\vspace*{.4cm}

\subsection{The proof of Theorem \ref{teoprob1}}
Some arguments of this proof were adapted from the articles \cite{Giovany,Olimpio}.
	
Maintaining the notation introduced in subsection \ref{periodic}, consider the energy functional $I_{A,V}:H^1_{A,V}(\R^N,\C)\to \R$ given by
\begin{align*}
I_{A,V}(u)=\frac{1}{2}{\|u\|^2}_{A,V} -\frac{1}{2\cdot 2_{\alpha}^*}D(u)-\frac{\lambda}{2p}B(u).
\end{align*}

We denote by $\mathcal{N}_{A,V}$ the Nehari Manifold related to $I_{A,V}$, that is,
\begin{align*}\mathcal{N}_{A,V}&=\left\{u\in H^1_{A,V}(\R^N,\C)\setminus\{0\}\,:\,{\|u\|}^2_{A,V}=D(u)+\lambda B(u)\right\},
\end{align*}
which is non-empty as a consequence of Theorem \ref{teoprop1}. As before, the functional $I_{A,V}$ satisfies the mountain pass geometry. Thus, there exists a sequence  $(u_n)\subset H^1_{A,V}(\mathbb{R}^N,\mathbb{C})$ such that
\[I'_{A,V}(u_n)\to 0\qquad\textrm{and}\qquad I_{A,V}(u_n)\to d_{\lambda},\]
where $d_{\lambda}$ is the minimax level, also characterized by
\begin{equation*}
d_{\lambda}=\inf_{u\;\in H^1_{A,V}(\R^N,\C)\setminus\{0\}}\max_{t\geq 0}I_{A,V}(tu)=\inf_{\mathcal{N}_{A,V}} I_{A,V}(u)>0.
\end{equation*}
	
We stress that, as a consequence of ($V_2$), we have $I_{A,V}(u)<J_{A,V_{\mathcal P}}(u)$ for all $u\in H^1_{A,V}(\R^N,\C)$.
	
The next lemma compares the levels $d_{\lambda}$ and $c_{\lambda}$.
	
\begin{lemma}\label{desiglevels}
The levels $d_{\lambda}$ and $c_{\lambda}$ verify the inequality
\[d_{\lambda}<c_{\lambda} <\frac{N+2-\alpha}{2(2N-\alpha)}(S_{A})^{\frac{2N-\alpha}{N+2-\alpha}}\]
for all $\lambda>0$.
\end{lemma}
\begin{proof}
Let $u$ be the ground state solution of problem \eqref{prop1.1} and consider $\bar{t}_u>0$ such that  $\bar{t}_u u \in \mathcal{N}_{A,V}$, that is
\begin{equation*}
0<d_{\lambda}\leq \sup_{t\geq 0} I_{A,V}(tu)=I_{A,V}(\bar{t}_u u).
\end{equation*}
		
It follows from $(V_2)$ that
\[0<d_{\lambda}\leq I_{A,V}(\bar{t}_u u)<J_{A,V_{\mathcal{P}}}(\bar{t}_u u)\leq \sup_{t\geq 0} J_{A,V_{\mathcal{P}}}(tu)= J_{A,V_{\mathcal{P}}}(u)=c_{\lambda}.\]

Therefore,
\[d_\lambda<c_{\lambda}.\]
The second inequality was already known.
$\hfill\Box$\end{proof}\vspace*{.2cm}

\textit{Proof of Theorem \ref{teoprob1}.}
Let $(u_n)$ be a  $(PS)_{d_{\lambda}}$ sequence for $I_{A,V}$.  As before,  $(u_n)$ is bounded in $H^1_{A,V}(\R^N,\C)$. Thus, there exists $u\in H^1_{A,V}(\R^N,\C)$ such that
\[u_n\rightharpoonup  u\ \ \textrm{in}\ \ H^1_{A,V}(\R^N,\C).\]
		
By the same arguments given in the proof of Theorem \ref{teoprop1}, $u$ is a ground state solution of problem \eqref{prob1}, if $u\neq 0$.
		
Following close \cite{Giovany}, we will show that $u=0$ cannot occur. Indeed, Lemma  \ref{lK} yields
\[\lim_{n\to \infty}\int_{\R^N} W|u_n|^2\;\dd x=0,
\]
since $W \in L^{\frac{N}{2}}(\R^N,\C)$ and $u_n\rightharpoonup 0 $ in $H^1_{A,V}(\R^N,\C)$. So,
\[|J_{A,V_{\mathcal{P}}}(u_n)-I_{A,V}(u_n)|=o_n (1)\]
showing that
\[J_{A,V_{\mathcal{P}}}(u_n)\to d_{\lambda}.\]
		
But, for $\varphi \in H^1_{A,V}(\R^N,\C)$ such that $\|\varphi\|_{A,V}\leq 1$, we have
\[
|(J'_{A,V_{\mathcal{P}}}(u_n)-I'_{A,V}(u_n))\cdot \varphi|\leq \bigg(\int_{\R^N} W|u_n|^2\; \dd x\bigg)^{\frac{1}{2}}=o_n(1).
\]
Thus,
\[J'_{A,V_{\mathcal{P}}}(u_n)=o_n (1)\]
		
Let $t_n>0$ such that $t_n u_n \in \mathcal{M}_{A,V_{\mathcal{P}}}$. Mimicking the argument found in \cite{Claudionor, Felmer, Rabinowitz,Willem}, it follows that $t_n\to 1$ as $n\to\infty$. Therefore,
\[c_{\lambda}\leq J_{A,V_{\mathcal{P}}}(t_n u_n)=J_{A,V_{\mathcal{P}}}(u_n)+o_n (1)=d_{\lambda}+o_n(1).\]
Letting $n\to +\infty$, we get
\[c_{\lambda}\leq d_{\lambda}\]
obtaining a contradiction with Lemma \ref{desiglevels}. This completes the proof of Theorem \ref{teoprob1}.
$\hfill\Box$\vspace{0.5 cm}
\section{The case $f(u)=|u|^{p-1} u$}
		\vspace{0.5 cm}
		
\subsection{The periodic problem} In this subsection we deal with problem \eqref{prop} for $f(u)$ as above, that is,
\begin{equation}\label{prop1.2}
-(\nabla+iA(x))^2u+  V_{\mathcal{P}}(x)u =\left(\frac{1}{|x|^{\alpha}}*|u|^{2_{\alpha}^*})\right)|u|^{2_{\alpha}^*-2} u + \lambda|u|^{p-1} u,
\end{equation}
where $1<p<2^{*} - 1$.

We observe that in this case the energy functional $J_{A,V_{\mathcal{P}}}$ is given by
\begin{equation*}
J_{A,V_{\mathcal{P}}}(u):=\frac{1}{2}\Vert u\Vert^2_{A,V_{\mathcal{P}}}-\frac{1}{2\cdot 2_{\alpha}^*}D(u)-\frac{\lambda}{p+1}\int_{\R^N}|u|^{p+1} \dd x,
\end{equation*}
where, as before
\[D(u)=\int_{\R^N}\left(\frac{1}{|x|^{\alpha}}*|u|^{2_{\alpha}^*}\right)|u|^{2_{\alpha}^*}\,\dd x=\int_{\R^N}\int_{\R^N}\frac{|u(x)^{2_{\alpha}^*}| |u(y)|^{2_{\alpha}^*}}{|x-y|^{\alpha}}\dd x\dd y.\]

By the Sobolev immersion \eqref{immersion} and the Hardy-Little\-wood-Sobolev inequality, we have that  $J_{A,V_{\mathcal{P}}}$ is well defined.

\begin{definition}
A function $u\in H^1_{A,V_{\mathcal{P}}}(\R^N,\C)$ is a weak solution of \eqref{prop1.2} if
\[\langle u,\varphi\rangle_{A,V_{\mathcal{P}}} -\mathfrak{Re}\int_{\mathbb{R}^N}\left(\frac{1}{|x|^{\alpha}}*|u|^{2_{\alpha}^*})\right)|u|^{2_{\alpha}^*-2}u\bar{\psi}\,\dd x-\lambda\;\mathfrak{Re}\int_{\mathbb{R}^N} |u|^{p-1}u\bar{\psi}\,\dd x =0\]
for all $\psi\in H^1_{A,V_{\mathcal{P}}}(\R^N,\C)$.
\end{definition}

As before, we see that critical points of $J_{A,V_{\mathcal{P}}}$ are weak solutions of \eqref{prop1.2} and
\begin{equation*}
J'_{A,V_{\mathcal{P}}}(u)\cdot u:=\Vert u\Vert_{A,V_{\mathcal{P}}}^{2}-D(u)-\lambda\|u\|^{p+1}_{p+1}.
\end{equation*}
We obtain that $J_{A,V_{\mathcal{P}}}$ satisfies the geometry of the mountain pass (see the proof of Lemma \ref{gpm}). 

As in Section \ref{Section3}, the mountain pass theorem without the PS condition yields a sequence $(u_n)\subset H^1_{A,V_{\mathcal{P}}}(\mathbb{R}^{N},\mathbb{C})$ such that
\begin{equation*}J'_{A,V_{\mathcal{P}}}(u_n)\to 0\qquad\textrm{and}\qquad J_{A,V_{\mathcal{P}}}(u_n)\to c_{\lambda},\end{equation*}
where $c_{\lambda}=\inf_{\alpha\in \Gamma}\max_{t\in [0,1]}J_{A,V_{\mathcal{P}}}(\gamma(t))$ and $\Gamma=\left\{\gamma\in C^1\left([0,1],H^1_{A,V_{\mathcal{P}}}(\mathbb{R}^{N},\mathbb{C})\right)\,:\,\gamma(0)=0,\, J_{A,V_{\mathcal{P}}}(\gamma(1))<0\right\}$.

Considering the Nehari manifold
$J_{A,V_{\mathcal{P}}}$
\begin{align*}\mathcal{M}_{A,V_{\mathcal{P}}}
&=\left\{u\in H^1_{A,V_{\mathcal{P}}}(\mathbb{R}^{N},\mathbb{C})\setminus\{0\}\,:\,\Vert u\Vert_{A,V_{\mathcal{P}}}^{2}= D(u) + \lambda \|u\|^{p+1}_{p+1}\right\},
\end{align*}
by proceeding as in the proof of Lemma \ref{equality} we obtain
\begin{lemma}\label{equality2}
There exists a unique $t_u=t_u (u)>0$  such that $t_u u\in \mathcal{M}_{A,V_{\mathcal{P}}} $ for all  $u\in H^1_{A,V_{\mathcal{P}}}(\mathbb{R}^{N},\mathbb{C})\setminus\{0\}$  and  $J_{A,V_{\mathcal{P}}}(t_u u)=\displaystyle \max_{t\geq 0} J_{A,V_{\mathcal{P}}}(tu)$. Moreover $c_{\lambda}=c^*_{\lambda}=c^{**}_{\lambda}$, where \[c_{\lambda}^{*}= \displaystyle  \inf_{u \;\in\; \mathcal{M}_{A,V_{\mathcal{P}}}} J_{A,V_{\mathcal{P}}}(u)\quad\textrm{and}\quad c_{\lambda}^{**}=\displaystyle \inf_{u\;\in\; H^1_{A,V_{\mathcal{P}}}(\R^N,\C)\setminus \{0\}}\max_{t\geq 0} J_{A,V_{\mathcal{P}}}(tu).\]
\end{lemma}

\begin{lemma}\label{lemmaconvderiv2}Suppose that $u_n\rightharpoonup u_0$ and consider 	
	\[B'(u_n)\cdot\psi=\mathfrak{Re}\int_{\mathbb{R}^N}|u|^{p-1}u\bar{\psi}\]
	and
	\[D'(u_n)\cdot\psi=\mathfrak{Re}\int_{\mathbb{R}^N}\bigg(\frac{1}{|x|^{\alpha}}*|u_n|^{2_{\alpha}^*})\bigg)|u_n|^{2_{\alpha}^*-2}u_n\bar{\psi}\]
	for $\psi\in C^\infty_c(\mathbb{R}^{N},\mathbb{C})$. Then $B'(u_n)\cdot \psi\to B'(u_0)\cdot \psi$ and  $D'(u_n)\cdot \psi\to D'(u_0)\cdot \psi$  	as $n\to\infty.$
\end{lemma}

\begin{lemma}\label{boundedsolut2}
	If $(u_n)\subset H^1_{A,V_{\mathcal{P}}}(\mathbb{R}^{N},\mathbb{C})$ is a  $(PS)_{\lambda}$ sequence for $J_{A,V_{\mathcal{P}}}$, then $(u_n)$ is bounded. In addition, if $u_n\rightharpoonup u$ weakly in $H^1_{A,V_{\mathcal{P}}}(\mathbb{R}^{N},\mathbb{C})$, as $n\to \infty,$ then $u$ is ground state solution for problem \eqref{prop1.2}.
\end{lemma}

\begin{lemma}\label{liminf2}
	If $(u_n)\subset H^1_{A,V_{\mathcal{P}}}(\R^N,\C)$ is a sequence $(PS)_{c_{\lambda}}$ for $J_{A,V_{\mathcal{P}}}$ such that
	\begin{equation*}
	u_n \rightharpoonup 0\quad \textrm{weakly in}\; H^1_{A,V_{\mathcal{P}}}(\R^N,\C) \;\textrm{ as}\; n\to\infty,
	\end{equation*}
	with
	\begin{equation*} c_{\lambda}<\frac{N+2-\alpha}{2(2N- \alpha)}S_{A}^{\frac{2N-\alpha}{N+2-\alpha}},
	\end{equation*}
	then there exists a sequence $(y_n)\in \R^N$ and constants $R,\theta>0$ such that
	\[\limsup_{n\to\infty}\int_{B_r(y_n)}|u_n|^2 \; \dd x\geq \theta,\]
	where $B_r(y)$ denotes the ball in $\R^N$ of center at $y$ and  radius $r>0$.
\end{lemma}
The proof of Lemmas \ref{lemmaconvderiv2}, \ref{boundedsolut2} and \ref{liminf2} is similar to that of Corollary \ref{lemmaconvderiv} Lemmas \ref{boundedsolut1} and \ref{conseqlions}, respectively.
\begin{lemma}\label{mainlemma2}
Let $1<p<2^* -1$ and $u_{\varepsilon}$ as defined in \eqref{defU}. Then,  there exists $\varepsilon$ such that
\begin{equation*}
\sup_{t\geq 0}J_{A,V_{\mathcal{P}}}(tu_{\varepsilon})<\frac{N+2-\alpha}{2(2N-\alpha)}(S_{A})^{\frac{2N-\alpha}{N+2-\alpha}}.
\end{equation*}
provided that either
\begin{enumerate}
\item [$(i)$] $3<p<5$, $N=3$ and $\lambda>0;$
\item [$(ii)$] $p>1$, $N\geq 4$ and $\lambda>0$;
\item [$(iii)$] $1<p\leq 3$, $N=3$ and $\lambda$ sufficiently large.
\end{enumerate}
\end{lemma}

\begin{proof} Consider, for the cases ($i$) and ($ii$) the function $f:[0,+\infty)\to \R$ defined by
\[f(t)=J_{A,V_{\mathcal{P}}}(tu_{\varepsilon})=\frac{t^2}{2}\|u_{\varepsilon}\|^2_{A,V_{\mathcal{P}}}-\frac{t^{2\cdot 2_{\alpha}^*}}{2\cdot 2_{\alpha}^*}D(u_{\varepsilon}) -\frac{\lambda t^{p+1}}{p+1}\|u_{\varepsilon}\|_{p+1}^{p+1}\]
and proceed as in the proof of Case 1, Lemma \ref{mainlemma}.

In the case of $1<p\leq 3$, $N=3$ and $\lambda$ sufficiently large, consider
$g_\lambda:[0,+\infty)\to\R$ defined by 
\[g_{\lambda} (t)= J_{A,V_{\mathcal{P}}}(t u_{\varepsilon})= \frac{t^2}{2}\int_{\R^N} \left[|\nabla u_{\varepsilon}|^2+\left(|A(x)|^2+V_{\mathcal{P}}(x)\right)|u_{\varepsilon}|^2\right]\dd x-\frac{1}{2\cdot 2_{\alpha}^*}t^{2\cdot 2_{\alpha}^*} D(u_{\varepsilon})-\frac{\lambda t^{p+1}}{p+1}\|u_{\varepsilon}\|_{p+1}^{p+1}\]
and proceed as in the proof of Case	2, Lemma \ref{mainlemma}.
$\hfill\Box$\end{proof}\vspace*{.4cm}

Similar to the proof of Theorem \ref{teoprop1}, we now state our result about the periodic problem \eqref{prop1.2}.

\begin{thm}\label{teoprop2}
Under the hypotheses already stated on $A$ and $\alpha$, suppose that $(V_1)$ is valid. Then problem \eqref{prop1.2} has at least one ground state solution if either
\begin{enumerate}
\item [$(i)$] $3<p<5$, $N=3$ and $\lambda>0;$
\item [$(ii)$] $p>1$, $N\geq 4$ and $\lambda>0$;
\item [$(iii)$] $1<p\leq 3$, $N=3$ and $\lambda$ sufficiently large.
\end{enumerate}
\end{thm}

\subsection{ Proof of Theorem \ref{teoprob2}}
Some arguments of this proof were adapted from the proof of Theorem \ref{teoprob1} below , that in turn  were adapted from  articles \cite{Giovany,Olimpio}.

Maintaining the notation already introduced, consider the functional $I_{A,V}:H^1_{A,V}(\R^N,\C)\to \R$ defined by
\begin{align*}
I_{A,V}(u):=\frac{1}{2}{\|u\|^2}_{A,V} -\frac{1}{2\cdot 2_{\alpha}^*}D(u)-\frac{\lambda}{p+1}\|u\|^{p+1}_{p+1}
\end{align*}
for all $u\in H^1_{A,V}(\R^N,\C)$.

We denote by $\mathcal{N}_{A,V}$ the Nehari Manifold related to $I_{A,V}$, that is,
\begin{align*}\mathcal{N}_{A,V}&=\left\{u\in H^1_{A,V}(\R^N,\C)\setminus\{0\}\,:\,{\|u\|}^2_{A,V}=D(u)+\lambda \|u\|^{p+1}_{p+1}\right\},
\end{align*}
which is non-empty as a consequence of Theorem \ref{teoprop2}. As before, the functional $I_{A,V}$ satisfies the mountain pass geometry. Thus, there exists a $(PS)_{d_{\lambda}}$ sequence  $(u_n)\subset H^1_{A,V}(\mathbb{R}^N,\mathbb{C})$, that is, a sequence satisfying
\[I'_{A,V}(u_n)\to 0\qquad\textrm{and}\qquad I_{A,V}(u_n)\to d_{\lambda},\]
where $d_{\lambda}$ is the minimax level, also characterized by
\begin{equation}\label{gpm12}
d_{\lambda}=\inf_{u\;\in H^1_{A,V}(\R^N,\C)\setminus\{0\}}\max_{t\geq 0}I_{A,V}(tu)=\inf_{\mathcal{N}_{A,V}} I_{A,V}(u)>0.
\end{equation}

As in the Section \ref{Section3}, we have $I_{A,V}(u)<J_{A,V_{\mathcal P}}(u)$ for all $u\in H^1_{A,V}(\R^N,\C)$ as a consequence of ($V_2$).

Similar to the proof of Lemma \ref{desiglevels} we have the following conclusion that shows as important inequality involving the levels $d_{\lambda}$ and $c_{\lambda}$, what completes the proof of Theorem \ref{teoprob2}.

\begin{lemma}\label{desiglevels2}
The levels $d_{\lambda}$ and $c_{\lambda}$ verify the inequality
\[d_{\lambda}<c_{\lambda} <\frac{N+2-\alpha}{2(N-\alpha)}(S_{A})^{\frac{2N-\alpha}{N+2-\alpha}}\]
for all $\lambda>0$.
\end{lemma}

\section{The case $f(u)=|u|^{2^{*}-2}u$}

\subsection{Proof of Theorem \ref{teoprob3}}

As observed by Gao and Yang \cite{BrezisN1}, the proof of Theorem \ref{teoprob3} is analogous to the proof of Theorem \ref{teoprob1}. The principal distinction is that the $(PS)_{c_\lambda}$ condition holds true below the level $\frac{1}{N}S^{\frac{N}{2}}$. It follows from \cite[Lemma 1.46]{Willem} that
\begin{equation*}
\displaystyle \int_{\R^N} |\nabla u_{\varepsilon}|^2 \dd x=S^{\frac{N}{2}}+O(\varepsilon^{N-2} )
\end{equation*}
and
\begin{equation*}
\displaystyle \int_{\R^N} | u_{\varepsilon}|^{2^{*}} \dd x=S^{\frac{N}{2}}+O(\varepsilon^N ).
\end{equation*}
So, we have
\begin{align*}
\displaystyle \sup_{t\geq 0}J_{A,V_{\mathcal{P}}}(t_{\varepsilon}u_{\varepsilon})
&<\frac{1}{N} S^{\frac{N}{2}}+ O(\varepsilon^{N-2})+ C_2\displaystyle \int_{\R^N}|u_{\varepsilon}(x)|^2 \dd x-C_3 \varepsilon^{2N-\alpha-(N-2)p}<\frac{1}{N}S^{\frac{N}{2}},
\end{align*}
since
\begin{equation*}
\lim_{\varepsilon\to 0}\varepsilon^{-(N-2)}\left(C_2\displaystyle \int_{\R^N}|u_{\varepsilon}(x)|^2 \dd x-C_3\varepsilon^{2N-\alpha-(N-2)p}\right)=-\infty.
\end{equation*}

Observe that the last result is a consequence of
\begin{equation*}
\lim_{\varepsilon\to 0}\varepsilon^{-(N-2)}\left(C_2\displaystyle \int_{B_{\delta}}|u_{\varepsilon}(x)|^2 \dd x-C_3\varepsilon^{2N-\alpha-(N-2)p}\right)=-\infty
\end{equation*}
and
\begin{equation*}\
C_2\displaystyle \int_{B_{2\delta}\setminus B_{\delta}}|u_{\varepsilon}(x)|^2 \dd x-C_3\varepsilon^{2N-\alpha-(N-2)p}=O(\varepsilon^{N-2}).
\end{equation*}
The rest of the proof is omitted here.
$\hfill\Box$\vspace*{.5cm}
	
\noindent\textrm{\textbf{Acknowledgements.}} The authors thank Prof. G. M. Figueiredo for many useful conversations and suggestions.

\end{document}